\documentclass[12pt]{article}
\usepackage{}
\usepackage{mathrsfs}
\usepackage{amsfonts}
\usepackage{amssymb,amsmath}
\usepackage{cite}
\usepackage{cases}
\usepackage{amsthm}
\usepackage{CJK}
\usepackage[usenames]{color}
\allowdisplaybreaks
\def\cl{\centerline}
\def\vs{\vspace*}
\def\ssc{\scriptscriptstyle}

\def\BZ{\mathbf{Z}}
\def\sl{\frak{sl}_2(\mathbb{C}_q)}

\def\wsl{\mbox{\footnotesize$\widetilde{\frak{sl}_2(\mathbb{C}_q)}$}}

\def\bz{{\bf 0}}
\def\bk{{\bf k}}
\def\bm{{\bf m}}
\def\bn{{\bf n}}
\def\LL{\mathfrak{L}}
\def\WL{\widetilde{\mathfrak{L}}}
\def\UU{{\cal U}}

\def\HH{\mathcal {H}}

\def\D{\Delta}
\def\r{\gamma}
\def\e{\epsilon}
\def\II{{\cal I}}
\def\JJ{{\cal J}}
\def\ot{\otimes}
\def\A{\mathscr{A}}

\def\Z{\mathbb{Z}{\ssc\,}}
\def\C{\mathbb{C}{\ssc\,}}

\numberwithin{equation}{section}
\newtheorem{theo}{Theorem}[section]
\newtheorem{defi}[theo]{Definition}
\newtheorem{coro}[theo]{Corollary}
\newtheorem{lemm}[theo]{Lemma}

\newtheorem{rema}[theo]{Remark}
\newtheorem{conv}[theo]{Convention}

\begin{document}
\begin{CJK*}{GBK}{song}
\cl{{\large\bf Quantizations of the extended affine
Lie algebra $\wsl$\,$^*$}
\footnote {$\!\!\!\!\!\!\!^*\,$Supported by NSF grant (No 10825101) and the China Postdoctoral Science Foundation Grant (No 201003326).}}\vs{6pt}
\cl{{Ying Xu$^{\,\dag}$},\ {Junbo Li$^{\,\dag,\,\ddag}$}}

\cl{\small
$^{\dag\,}$Wu Wen-Tsun Key Laboratory of \vs{-2pt}Mathematics}
\cl{\small  University of Science and
Technology of China, Hefei 230026, China}

\cl{\small $^{\ddag\,}$School of Mathematics and Statistics, Changshu Institute of Technology, Changshu 215500, China}

\cl{\small E-mail:
xying@mail.ustc.edu.cn,\ \ sd\_junbo@163.com}\vs{5pt}

{\small
\parskip .005 truein
\baselineskip 3pt \lineskip 3pt

\noindent{{\bf Abstract.}\ \
The extended affine Lie algebra $\wsl$ is quantized from three different points of view in this paper, which produces three noncommutative and noncocommutative Hopf algebra
structures, and yield other three quantizations by an isomorphism of $\wsl$ correspondingly. Moreover, two of these quantizations can be restricted to the extended affine Lie algebra $\sl$.

\vs{2pt}

}\vs{5pt}

}
\parskip .001 truein\baselineskip 6pt \lineskip 6pt
\vs{10 pt}

\noindent{\bf1. \
Introduction}\setcounter{section}{1}\setcounter{equation}{0}
\\[5pt]

During the investigation of quantum groups, V.Drinfel'd introduced the notion of  Lie bialgebras \cite{D} in 1983. Quantization of Lie algebras and bialgebras is an important way to produce new quantum groups. Quantizations by twists act basically for constructing new quantized enveloping algebras. A universal and functional quantization of Lie bialgebras was developed in \cite{EK,EK2} employing the Tannaka-Krein approach, from which a quantization of any finite dimensional Lie bialgebra defined over a field of characteristic zero (see \cite{EK}) was constructed. Although a general method for twisting both the product and coproduct of a bialgebra does not appear, it is possible
to twist the corresponding coproduct in such a way that it remains compatible with its original multiplication, unit, and counit (see \cite{GZh}). In this paper, we shall concentrate on the quantization being assort to the so-called Drinfel'd twist of the extended affine Lie algebra (EALA) $\wsl$, whose Lie bialgebra structures were determined in \cite{XL}. The EALA $\sl$ was first introduced in \cite{HT} in the sense of quasi-simple Lie algebras and systematically investigated in \cite{AABGP}. Since then, the representation and structure theory of such Lie algebras have been attentively studied (see \cite{AG,B,BGK,BGKE,BL,G1,G2,GZ,LwS,N} and the references therein).

We now introduce the Lie algebra considered in this paper. Denote $\Z$, $\Z^*$, $\C$
the sets of all integers, nonzero integers, complex numbers respectively. Let $\bz=(0,0)$, $\BZ=\Z\times\Z$, $\BZ^*=\Z^*\times\Z^*$. For any $\bm=(m_1,m_2)\in\BZ$, $\bk=(k_1,k_2)\in\BZ^*$, introduce the following elements of $\LL=\sl$:
\begin{eqnarray*}
&&e_{\bm}=
E_{12}x^{m_1}y^{m_2},\ \ \ f_{\bm}=
E_{21}x^{m_1}y^{m_2},\\
&&d=E_{11}-E_{22},\ \ \ \ \ \  g_{\bk}=E_{11}x^{k_1}y^{k_2},\ \ \ \ \ \
h_{\bk}=E_{22}x^{k_1}y^{k_2},
\end{eqnarray*}
which form a basis of $\LL$ with the following relations:
\begin{eqnarray*}
&&[e_{\bm},e_{\bm'}]=[f_{\bm},f_{\bm'}]=[d,h_{\bk}]=[d,g_{\bk}]=[g_{\bk},h_{\bk'}]=0,\\
&&[g_{\bk},e_{\bm}]=q^{k_2m_1}e_{\bk+\bm},\ \ [h_{\bk},e_{\bm}]=-q^{k_1m_2}e_{\bk+\bm},\ \ [d,e_{\bm}]=2e_{\bm},\\
&&[h_{\bk},f_{\bm}]=q^{k_2m_1}f_{\bk+\bm},\ \ [g_{\bk},f_{\bm}]=-q^{k_1m_2}f_{\bk+\bm},\ \ [d,f_{\bm}]=-2f_{\bm},\\
&&[e_{\bm},f_{\bm'}]=
\Big\{\begin{array}{ll}
q^{m_2m'_1}g_{\bm+\bm'}-q^{m'_2m_1}h_{\bm+\bm'}&{\rm if}\ \ \bm+\bm'\neq\bz,\\[6pt]
q^{m_2m'_1}d&{\rm if}\ \ \bm+\bm'=\bz,
\end{array}
\\
&&[g_{\bk},g_{\bk'}]=
\Big\{\begin{array}{ll}
(q^{k_2k'_1}-q^{k'_2k_1})g_{\bk+\bk'}&{\rm if}\,\ \ \bk+\bk'\neq\bz,\\[6pt]
0&{\rm if}\,\ \ \bk+\bk'=\bz,
\end{array}
\\
&&[h_{\bk},h_{\bk'}]=
\Big\{\begin{array}{ll}
(q^{k_2k'_1}-q^{k'_2k_1})h_{\bk+\bk'}&{\rm if}\,\ \ \bk+\bk'\neq\bz,\\[6pt]
0&{\rm if}\,\ \ \bk+\bk'=\bz.
\end{array}
\end{eqnarray*}
Then $\LL=\oplus_{\bm}\LL_{\bm}$, where $\LL_{\bz}=\C e_{\bz}\oplus \C f_{\bz}+\C d$, $\LL_{\bk}=\C e_{\bk}\oplus \C f_{\bk}\oplus \C g_{\bk}\oplus \C h_{\bk}$. Introduce two degree derivations $d_1$ and $d_2$ on $\LL$:
$$[d_1,L]=m_1L,\ \,[d_2,L]=m_2L\mbox{ \ for $L\in \LL_{\bm}$\ \ and }\ \ [d_1,d_2]=[d,d_1]=[d,d_2]=0.$$
Then we arrive at the EALA $\wsl=\LL\oplus\C d_1\oplus\C d_2$ considered in this paper and denoted by $\WL$ for convenience. Also $\WL$ is $\BZ$-graded: $\WL=\oplus_{\bm\in\BZ}\WL_{\bm}$ with $\WL_{\bk}=\LL_{\bk}$ for $\bk\in\BZ^*$ and $\WL_{\bz}=\LL_{\bz}\oplus\C d_1\oplus\C d_2$. Denote the universal enveloping algebra of the Lie algebra $\wsl$ by $\UU(\wsl)$.

For $\bm=(m_1,m_2),\,\bn=(n_1,n_2)$, $i\in\Z$, introduce the following notations that will be referred to in the main theorem:
\begin{eqnarray*}
&&\gamma_{i,\bm}^y=
\begin{cases}
1,&i=0, y=h,g,f,e,\\[3pt]
0,&i>0, y=h,\\[3pt]
(-1)^i\mbox{$\prod\limits_{p=1}^i$} (q^{n_2(m_1+(p-1)n_1)}-q^{n_1(m_2+(p-1)n_2)}),&i>0, y=g,\\[3pt]
\mbox{$\prod\limits_{p=1}^i$} q^{n_1(m_2+(p-1)n_2)},&i>0, y=f,\\[3pt]
(-1)^i\mbox{$\prod\limits_{p=1}^i$} q^{n_2(m_1+(p-1)n_1)},&i>0, y=e,\\[3pt]
\end{cases}\\[3pt]
&&\alpha_{y_{\bm}}=
\begin{cases}
0,&y=e,\\[3pt]
q^{m_2n_1},&y=g,\\[3pt]
-q^{m_1n_2},&y=h.
\end{cases}\ \ \ \ \ \ \
(1-Et)^{\delta_{y,e}}=
\begin{cases}
1-Et,&y=e,\\[3pt]
1,&y=g,\\[3pt]
1,&y=h.
\end{cases}\\
&&s_{\bm}=q^{n_2m_1+n_1m_2+n_1n_2}.
\end{eqnarray*}
The main result of this paper can be formulated as the following theorem.
\begin{theo}\label{main}
There exist some noncommutative and noncocommutative Hopf algebra
structures $(\UU (\wsl )[[t]], \mu, \tau, \D, \e,S )$ on $\UU
(\wsl)[[t]]$ over $\C[[t]]$, which preserve the product and the counit
of $\UU (\wsl )[[t]]$, admitting the following corresponding coproducts and antipodes:
\begin{item}
  \item {\rm(1)} For $T=\mbox{$\sum\limits_{i=1}^2$}x_id_i$, $E=g_{\bn}$ with $[T,E]=E$ and $r=x_1m_1+x_2m_2$, $y=e,f,g,h$, $i=1,2$,
\begin{eqnarray*}
&\D(y_{\bm})\!\!\!\!&=y_{\bm}\otimes(1-Et)^r+\mbox{$\sum\limits_{j=0}^{\infty}$}
\frac{\gamma_{j,\bm}^y}{j!}T^{<j>}\otimes(1-Et)^{-j}y_{\bm+j\bn}t^j,\\
&S(y_{\bm})\!\!\!\!&=\mbox{$\sum\limits_{j=0}^{\infty}$}\frac{(-1)^{j+1}\gamma_{y_{\bm}}^j}{j!}(1-Et)^{-r}y_{\bm+j\bn}T_{1-c}^{<j>}t^j,\\
&\D(d_i)\!\!\!\!&=d_i\otimes1+1\otimes d_i-n_iT\otimes1+n_iT\otimes(1-Et)^{-1},\\
&\!\!\D(d)\!\!\!\!&=d\otimes1+1\otimes d,\ \ S(d)=-d,\ \
S(d_i)=-d_i+n_iTEt.
\end{eqnarray*}
  \item {\rm(2)} For $T=\mbox{$\sum\limits_{i=1}^2$}x_id_i$, $E=e_{\bn}$,  with $[T,E]=E$ and $r=x_1m_1+x_2m_2$, $y=e,g,h$, $i=1,2$,
\begin{eqnarray*}
&\D(y_{\bm})\!\!\!\!&=y_{\bm}\otimes(1-Et)^r+
1\otimes y_{\bm}+\alpha_{y_{\bm}}T\otimes (1-Et)^{-1}e_{\bm+\bn}t,\\
&\D(f_{\bm})\!\!\!\!&=
\begin{cases}
q^{m_2n_1}T\otimes (1-Et)^{-1}h_{\bm+\bn}t-q^{m_1n_2}T\otimes (1-Et)^{-1}g_{\bm+\bn}t\\
+f_{\bm}\otimes(1-Et)^r+1\otimes f_{\bm}-s_{\bm}T^{<2>}\otimes (1-Et)^{-2}e_{\bm+2\bn}t^2,&\bm+\bn\neq0,\\ \\
f_{-\bn}\otimes(1-Et)^{-1}+
1\otimes f_{-\bn}-q^{-n_2n_1}T\otimes (1-Et)^{-1}dt\\
-q^{-n_1n_2}T^{<2>}\otimes (1-Et)^{-2}Et^2,&\bm+\bn=0,
\end{cases}\\
&\!\D(d)\!\!\!\!&=d\otimes1+1\otimes d+2T\otimes (1-Et)^{-1}Et,\\ &\D(d_i)\!\!\!\!&=d_i\otimes1+1\otimes d_i-n_iT\otimes1+n_iT\otimes(1-Et)^{-1},\\
&S(y_{\bm})\!\!\!\!&=-(1-Et)^{r}y_{\bm}+\alpha_{y_{\bm}}(1-Et)^{r}e_{\bm+\bn}T_{1}t,\\
&S(f_{\bm})\!\!\!\!&=
\begin{cases}
q^{m_2n_1}(1-Et)^{r}h_{\bm+\bn}T_{1}t
-q^{m_1n_2}(1-Et)^{r}g_{\bm+\bn}T_{1}t\\
-(1-Et)^{r}f_{\bm}+s_{\bm}(1-Et)^{r}e_{\bm+2\bn}T_{1}^{<2>}t^2,&\bm+\bn\neq0,\\ \\
q^{-n_1n_2}(1-Et)^{-1}ET_1^{<2>}t^2-(1-Et)^{-1}f_{-\bn}\\
-q^{-n_1n_2}(1-Et)^{-1}dT_{1}t,&\bm+\bn=0,
\end{cases}\\
&S(d)\!\!\!\!&=-d+2ET_1t,\ \ S(d_i)=-d_i+n_iTEt.
\end{eqnarray*}
  \item {\rm(3)} For $T=\frac{1}{2}d$, $E=e_{\bn}$ and $y=e,g,h$, $i=1,2$,
\begin{eqnarray*}
&\D(y_{\bm})\!\!\!\!&=y_{\bm}\otimes(1-Et)^{\delta_{y,e}}+
1\otimes y_{\bm}+\alpha_{y_{\bm}}T\otimes (1-Et)^{-1}e_{\bm+\bn}t,\\
&\D(f_{\bm})\!\!\!\!&=
\begin{cases}
q^{m_2n_1}T\otimes (1-Et)^{-1}h_{\bm+\bn}t-q^{m_1n_2}T\otimes (1-Et)^{-1}g_{\bm+\bn}t\\
+f_{\bm}\otimes(1-Et)^{-1}
+1\otimes f_{\bm}-s_{\bm}T^{<2>}\otimes (1-Et)^{-2}e_{\bm+2\bn}t^2,&\bm+\bn\neq0,\\ \\
f_{-\bn}\otimes(1-Et)^{-1}-q^{-n_2n_1}T\otimes(1-Et)^{-1} dt\\
+1\otimes f_{-\bn}-q^{-n_1n_2}T^{<2>}\otimes (1-Et)^{-2}t^2,&\bm+\bn=0,
\end{cases}\\
&\D(d)\!\!\!\!&=d\otimes1+1\otimes d+2T\otimes (1-Et)^{-1}Et,\\
&\D(d_i)\!\!\!\!&=d_i\otimes1+1\otimes d_i+n_iT\otimes(1-Et)^{-1}-n_iT\otimes1,\\
&S(y_{\bm})\!\!\!\!&=-(1-Et)^{\delta_{y,e}}y_{\bm}
+\alpha_{y_{\bm}}e_{\bm+\bn}T_{1}t,\\
&S(f_{\bm})\!\!\!\!&=
\begin{cases}
q^{m_2n_1}(1-Et)^{-1}h_{\bm+\bn}T_{1}t
-q^{m_1n_2}(1-Et)^{-1}T_1g_{\bm+\bn}t\\
-(1-Et)^{-1}f_{\bm}
+s_{\bm}(1-Et)^{-1}e_{\bm+2\bn}T_{1}^{<2>}t^2,&\bm+\bn\neq0,\\ \\
q^{-n_1n_2}(1-Et)^{-1}ET_{1}^{<2>}t^2-(1-Et)^{-1}f_{-\bn}\\
-q^{-n_1n_2}(1-Et)^{-1}dT_1t,&\bm+\bn=0,
\end{cases}\\
&S(d)\!\!\!\!&=-d+2ET_1t,\ \ \ \ S(d_i)=-d_i+n_iTEt.
\end{eqnarray*}
\end{item}
\end{theo}

For $\bm=(m_1,m_2),\,\bn=(n_1,n_2)$, $i\in\Z$, introduce the following notations that will be referred to in the following corollary:
\begin{eqnarray*}
&&\eta_{i,\bm}^y=
\begin{cases}
1,&i=0, y=h,g,f,e,\\[3pt]
0,&i>0, y=g,\\[3pt]
(-1)^i\mbox{$\prod\limits_{p=1}^i$} (q^{n_2(m_1+(p-1)n_1)}-q^{n_1(m_2+(p-1)n_2)}),&i>0, y=h,\\[3pt]
\mbox{$\prod\limits_{p=1}^i$} q^{n_1(m_2+(p-1)n_2)},&i>0, y=e,\\[3pt]
(-1)^i\mbox{$\prod\limits_{p=1}^i$} q^{n_2(m_1+(p-1)n_1)},&i>0, y=f,\\[3pt]
\end{cases}\\[3pt]
&&\beta_{y_{\bm}}=
\begin{cases}
0,&y=f,\\[3pt]
-q^{m_1n_2},&y=g,\\[3pt]
q^{m_2n_1},&y=h,
\end{cases}\ \ \ \ \ \
(1-Et)^{\delta_{y,f}}=
\begin{cases}
1-Et,&y=f,\\[3pt]
1,&y=g,\\[3pt]
1,&y=h.
\end{cases}
\end{eqnarray*}
Combining Theorem \ref{main} and the following involution of $\wsl$:
$$\tau:\,e_{\bm}\leftrightarrow f_{\bm},\ \ g_{\bn}\leftrightarrow h_{\bn},\ \ d\leftrightarrow -d,\ \ d_{i}\leftrightarrow d_{i},\ \ \forall\,\bm\in\BZ,\ \bn\in\BZ^*,\ i=1,2,$$
we can immediately derive the following corollary, which presents other  three quantizations of $\wsl$.

\begin{coro}\label{coro}
There exist some noncommutative and noncocommutative Hopf algebra
structures $(\UU (\wsl )[[t]], \mu, \tau, \D, \e,S )$ on $\UU
(\wsl)[[t]]$ over $\C[[t]]$, which preserve the product and the counit
of $\UU (\wsl )[[t]]$, admitting the following corresponding coproducts and antipodes:
\begin{item}
  \item {\rm(1)} For $T=\mbox{$\sum\limits_{i=1}^2$}x_id_i$, $E=h_{\bn}$ with $[T,E]=E$ and $r=x_1m_1+x_2m_2$, $y=e,f,h,g$, $i=1,2$,
\begin{eqnarray*}
&\D(y_{\bm})\!\!\!\!&=y_{\bm}\otimes(1-Et)^r+\mbox{$\sum\limits_{j=0}^{\infty}$}
\frac{\eta_{j,\bm}^y}{j!}T^{<j>}\otimes(1-Et)^{-j}y_{\bm+j\bn}t^j,\\
&S(y_{\bm})\!\!\!\!&=\mbox{$\sum\limits_{j=0}^{\infty}$}\frac{(-1)^{j+1}
\eta_{j,\bm}^y}{j!}(1-Et)^{-r}y_{\bm+j\bn}T_{1-c}^{<j>}t^j,\\
&\D(d_i)\!\!\!\!&=d_i\otimes1+1\otimes d_i-n_iT\otimes1+n_iT\otimes(1-Et)^{-1},\\
&\!\!\D(d)\!\!\!\!&=d\otimes1+1\otimes d,\ \ S(d)=-d,\ \
S(d_i)=-d_i+n_iTEt.
\end{eqnarray*}
  \item {\rm(2)} For $T=\mbox{$\sum\limits_{i=1}^2$}x_id_i$, $E=f_{\bn}$,  with $[T,E]=E$ and $r=x_1m_1+x_2m_2$, $y=f,g,h$, $i=1,2$,
\begin{eqnarray*}
&\D(y_{\bm})\!\!\!\!&=y_{\bm}\otimes(1-Et)^r+
1\otimes y_{\bm}+\beta_{y_{\bm}}T\otimes (1-Et)^{-1}f_{\bm+\bn}t,\\
&\D(e_{\bm})\!\!\!\!&=
\begin{cases}
q^{m_2n_1}T\otimes (1-Et)^{-1}g_{\bm+\bn}t-q^{m_1n_2}T\otimes (1-Et)^{-1}h_{\bm+\bn}t\\
+e_{\bm}\otimes(1-Et)^r+1\otimes e_{\bm}-s_{\bm}T^{<2>}\otimes (1-Et)^{-2}f_{\bm+2\bn}t^2,&\bm+\bn\neq0,\\ \\
e_{-\bn}\otimes(1-Et)^{-1}+
1\otimes e_{-\bn}+q^{-n_2n_1}T\otimes (1-Et)^{-1}dt\\
-q^{-n_1n_2}T^{<2>}\otimes (1-Et)^{-2}Et^2,&\bm+\bn=0,
\end{cases}\\
&\D(d)\!\!\!\!&=d\otimes1+1\otimes d-2T\otimes (1-Et)^{-1}Et,\\
&\D(d_i)\!\!\!\!&=d_i\otimes1+1\otimes d_i-n_iT\otimes1+n_iT\otimes(1-Et)^{-1},\\
&S(y_{\bm})\!\!\!\!&=-(1-Et)^{r}y_{\bm}+\beta_{y_{\bm}}
(1-Et)^{r}f_{\bm+\bn}T_{1}t,\\
&S(e_{\bm})\!\!\!\!&=
\begin{cases}
q^{m_2n_1}(1-Et)^{r}g_{\bm+\bn}T_{1}t
-q^{m_1n_2}(1-Et)^{r}h_{\bm+\bn}T_{1}t\\
-(1-Et)^{r}e_{\bm}+s_{\bm}(1-Et)^{r}f_{\bm+2\bn}T_{1}^{<2>}t^2,&\bm+\bn\neq0,\\ \\
q^{-n_1n_2}(1-Et)^{-1}ET_1^{<2>}t^2-(1-Et)^{-1}e_{-\bn}\\
+q^{-n_1n_2}(1-Et)^{-1}dT_{1}t,&\bm+\bn=0,
\end{cases}\\
&S(d)\!\!\!\!&=-d-2ET_1t,\ \ S(d_i)=-d_i+n_iTEt.
\end{eqnarray*}
\item {\rm(3)} For $T=\frac{1}{2}d$, $E=f_{\bn}$ and $y=f,g,h$, $i=1,2$,
\begin{eqnarray*}
&\D(y_{\bm})\!\!\!\!&=y_{\bm}\otimes(1-Et)^{\delta_{y,f}}+
1\otimes y_{\bm}+\beta_{y_{\bm}}T\otimes (1-Et)^{-1}e_{\bm+\bn}t,\\
&\D(e_{\bm})\!\!\!\!&=
\begin{cases}
q^{m_2n_1}T\otimes (1-Et)^{-1}g_{\bm+\bn}t-q^{m_1n_2}T\otimes (1-Et)^{-1}h_{\bm+\bn}t\\
+e_{\bm}\otimes(1-Et)^{-1}
+1\otimes e_{\bm}-s_{\bm}T^{<2>}\otimes (1-Et)^{-2}f_{\bm+2\bn}t^2,&\bm+\bn\neq0,\\ \\
e_{-\bn}\otimes(1-Et)^{-1}+q^{-n_2n_1}T\otimes(1-Et)^{-1} dt\\
+1\otimes e_{-\bn}-q^{-n_1n_2}T^{<2>}\otimes (1-Et)^{-2}t^2,&\bm+\bn=0,
\end{cases}\\
&\!\D(d)\!\!\!\!&=d\otimes1+1\otimes d-2T\otimes (1-Et)^{-1}Et,\\
&\D(d_i)\!\!\!\!&=d_i\otimes1+1\otimes d_i+n_iT\otimes(1-Et)^{-1}-n_iT\otimes1,\\
&S(y_{\bm})\!\!\!\!&=-(1-Et)^{\delta_{y,f}}y_{\bm}+\beta_{y_{\bm}}f_{\bm+\bn}T_{1}t,\\
&S(e_{\bm})\!\!\!\!&=
\begin{cases}
q^{m_2n_1}(1-Et)^{-1}g_{\bm+\bn}T_{1}t
-q^{m_1n_2}(1-Et)^{-1}T_1h_{\bm+\bn}t\\
-(1-Et)^{-1}e_{\bm}
+s_{\bm}(1-Et)^{-1}f_{\bm+2\bn}T_{1}^{<2>}t^2,&\bm+\bn\neq0,\\ \\
q^{-n_1n_2}(1-Et)^{-1}ET_{1}^{<2>}t^2-(1-Et)^{-1}f_{-\bn}\\
+q^{-n_1n_2}(1-Et)^{-1}dT_1t,&\bm+\bn=0,
\end{cases}\\
&S(d)\!\!\!\!&=-d-2ET_1t,\ \ \ \ S(d_i)=-d_i+n_iTEt.
\end{eqnarray*}
\end{item}
\end{coro}

\begin{conv} \rm
If an undefined term appears in an expression, we always treat it as zero, e.g., $g_{\bz}=h_{\bz}=0$.
\end{conv}

\begin{rema} \rm
(1)\ \ We have in fact exhausted all the possibilities of Drinfel'd twist quantizations based on the ``usual'' noncommutative 2-dimensional Lie subalgebras $\{T,E\}$ of $\wsl$ up to scalar multiplications (the ``usual'' means that one of the two elements, i.e., $T$, is in the Cartan subalgebra of $\wsl$). This is also the main reason why we present 6 quantizations above.
(We are currently engaging in an investigation of the Drinfel'd twist quantizations based on ``unusual'' choices of noncommutative 2-dimensional Lie subalgebras of $\wsl$. However, it seems to us that heavy difficulties appear and that new techniques should be introduced during the process of such attempt.)\\

(2)\ \ Although the Lie bialgebra structures on the affine Lie algebra $\sl$ have not been determined yet, we may obtain two quantizations of the affine Lie algebras $\sl$ by restricting the third quantizations in Theorem \ref{main} and Corollary \ref{coro} to $\sl$ by taking $d_1=d_2=0$.
\end{rema}

\vskip12pt

\noindent{\bf2. Definition and preliminary
results}\setcounter{section}{2}\setcounter{theo}{0}\setcounter{equation}{0}

\vskip6pt

We first recall some basic concepts and results based on a unital $\C$-algebra $\mathscr{A}$. For any element $x \in \mathscr{A}, a\in \C,r\in\Z_+,$ set $x^{<r>} = x_0 ^{<r>}$, $x^{[r]} = x_0 ^{[r]}$, where
\begin{eqnarray*}
&x_a ^{<r>}&\!\!\!\!=(x+a)(x+a+1)\cdots(x+a+r-1),
\\[6pt]
&x_a ^{[r]}&\!\!\!\!= (x+a)(x+a-1)\cdots(x+a-r+1).
\end{eqnarray*}

For convenient, set $x_a^{<0>}=x_a^{[0]}=1$. The following lemma can be found in \cite{G} or \cite{GZh}.

\begin{lemm}\label{lem2} For any $x\in\mathscr{A}$, $a, d \in \C $ and $r, s, m\in \Z_+$, one has
\begin{eqnarray}
&&x_a^{<r+s>}=x_a^{<r>}x_{a+r}^{<s>},\ \ \ x_a^{[r+s]}=x_a^{[r]}x_{a-r}^{[s]},\ \ \ x_a ^{[r]}=x_{a-r+1}^{<r>},\label{x<>[]}\\[6pt]
&&\mbox{$\sum\limits_{r+s=m}$}\frac{(-1)^s}{r!s!}x_a ^{[r]}x_d^{<s>}=\dbinom{a-d}{m},\ \ \ \
\mbox{$\sum\limits_{r+s=m}$}\frac{(-1)^s}{r!s!}x_a^{[r]}x_{d-r}^{[s]}=\dbinom{a-d+m-1}{m},
\nonumber
\end{eqnarray}
where the binomial coefficient
\begin{eqnarray*}
\binom{a}{d}=
\begin{cases}
\frac{a(a-1)\cdots(a-d+1)}{d!}, &a\geq d, \\
0, &a<d.
\end{cases}
\end{eqnarray*}
\end{lemm}

It is known that there is a natural Hopf algebra structure on the universal enveloping algebra of the Lie algebra $\wsl$, denoted by $(\UU(\wsl),\mu,\tau, \D_0,S_0,\e_0)$ with
\begin{eqnarray*}
\D_0(L_{\bm})=L_{\bm}\ot 1+1\ot L_{\bm},\ \ \ S_0(L_{\bm})=-L_{\bm},\ \ \ \e_0(L_{\bm})=0,\ \ \ \forall\,L_{\bm}\in\WL_{\bm}.
\end{eqnarray*}
Then a deformation of $\UU (\wsl)$ is a topologically free $\C[[t]]$-algebra
$\UU (\wsl)[[t]]$, i.e., it is an associative $\C$-algebra of formal
power series with coefficients in $\UU (\wsl)$ such that $\UU (\wsl
)[[t]]/t\UU (\wsl )[[t]] \cong \UU (\wsl )$. Naturally, $\UU
(\wsl)[[t]]$ is equiped with a Hopf algebra structure induced from
$\UU (\wsl)$. We also denote it by $(\UU (\wsl))[[t]], \mu, \tau, \D_0
, \e_0, S_0 )$.

\begin{defi} \rm Let $(\A, \mu, \tau, \D, S,\e)$ be a
Hopf algebra over a commutative ring. A {\em Drinfel'd twist} $\II$
on $\A$ is an  invertible element of  $ \A \otimes \A$ such that
\begin{eqnarray*}
&\!\!\!\!\!\!\!\! &(\II \otimes 1) (\D \otimes Id ) (\II ) = (1
\otimes \II) (1 \otimes \D ) (\II ),\\
 \label{2.11}&\!\!\!\!\!\!\!\! & (\e \otimes Id )
(\II ) = 1 \otimes 1 = (Id \otimes \e ) (\II).
\end{eqnarray*}
\end{defi}

It is known that the Drinfel'd twists pay an important role in constructing a new Hopf algebra. We shall employ the following Lemma (see \cite{D}) to complete the quantization of $\wsl$ (also see \cite{CS,LS,SSW}).

\begin{lemm} Let $(\A, \mu, \tau, \D_0, \e_0, S_0)$ be a
Hopf algebra over a commutative ring, $\II$ a Drinfel'd twist on
$\A$. Then
\begin{itemize}
 \item[{\rm(1)}] $u = \mu (Id \otimes S ) (\II )$ is invertible
in $\A \otimes \A$ with $u^{-1} = \mu (S \otimes Id ) (\II) .$
 \item[{\rm(2)}] The algebra $(\A, \mu, \tau, \D, \e, S )$ is a new Hopf
 algebra where
 \begin{equation}
 \D = \II \D_0 \II^{-1} , \ \ \,\e=\e_0,\ \ \ S = u S_0 u^{-1}.\label{equa-D-D0}
 \end{equation}
\end{itemize}
\end{lemm}

For formal variable $t$, and $c\in\C$, $T,E\in\wsl$ with
$[T,E]=E$, denote
\begin{eqnarray*}
&\II_{c}=\mbox{$\sum\limits_{i=0}^{\infty}$}\frac{(-1)^i}{i!}T_c^{[i]}\otimes E^it^i,\ \ &\JJ_{c}=\mu (Id\otimes S_0)(\II_c),\\
&I_{c}=\mbox{$\sum\limits_{i=0}^{\infty}$}\frac{1}{i!}T_c^{<i>}\otimes E^it^i,\ \ &J_{c}=\mu (S_0\otimes Id) (I_c).
\end{eqnarray*}
The following lemma also holds according to the corresponding lemma in \cite{CS,LS,SSW}.
\begin{lemm}
\begin{eqnarray}
&&\!\!\!\!\!\!\!\!\!\!\!\!\!\!\!
\JJ_c=\mbox{$\sum\limits_{i=0}^{\infty}$}\frac{1}{i!}T_c^{[i]}E^it^i,\ \
J_c=\mbox{$\sum\limits_{i=0}^{\infty}$}\frac{(-1)^i}{i!}T_{-c}^{[i]}E^it^i,\label{equa-J}\\
&&\!\!\!\!\!\!\!\!\!\!\!\!\!\!\!
\D_0(T^{[m]})=\mbox{$\sum\limits_{i=0}^{m}$}\dbinom{m}{i}T^{[i]}_{-c}\otimes T_c^{[m-i]},\\
&&\!\!\!\!\!\!\!\!\!\!\!\!\!\!\!
\II_c I_d=1\otimes (1-Et)^{(c-d)},\ \, \JJ_cJ_d=(1-Et)^{-(c+d)}.\label{equa-II-JJ}
\end{eqnarray}
In particular, $\II_c,\,I_c,\,\JJ_c,\,J_c$ are invertible elements with $\II_c^{-1}=I_c,\ \JJ_c^{-1}=J_{-c}$ and $\II_0$ is a Drinfel'd twist of $(\UU (\wsl))[[t]], \mu, \tau, \D_0, \e_0, S_0 )$.
\end{lemm}

\vskip12pt

\noindent{\bf3. Proof of Theorem\ref{main}\,(1)}\setcounter{section}{3}\setcounter{theo}{0}\setcounter{equation}{0}

\vskip6pt

In this section, we take $T=x_1d_1+x_2d_2$ for some $x_1,x_2\in\C$, $\bn=(n_1,n_2)\in\BZ$ and $E=g_{\bn}$ such that $[T,E]=E$. It is easy to see that $x_1n_1+x_2n_2=1$. For $\bm=(m_1,m_2)\in\BZ$, $r=x_1m_1+x_2m_2$, denote
\begin{eqnarray*}
\r_{i,\bm}^g\!\!\!&=&\!\!\!(-1)^i\mbox{$\prod\limits_{p=1}^i$} (q^{n_2(m_1+(p-1)n_1)}-q^{n_1(m_2+(p-1)n_2)}),\\
\r_{i,\bm}^f\!\!\!&=&\!\!\!\mbox{$\prod\limits_{p=1}^i$} q^{n_1(m_2+(p-1)n_2)},\ \ \  \r_{i,\bm}^e=(-1)^i\mbox{$\prod\limits_{p=1}^i$} q^{n_2(m_1+(p-1)n_1)}.
\end{eqnarray*}
\begin{lemm}
The following identities hold in $\UU(\wsl)$ (where $l_{\bm}\in\WL_{\bm}$):
\begin{eqnarray}
&&l_{\bm}T_c^{[i]}=T_{c-r}^{[i]}l_{\bm},\ \ l_{\bm}T_c^{<i>}=T_{c-r}^{<i>}l_{\bm}, \label{Td12Eg-Tl}\\
&&ET_c^{[i]}=T_{c-1}^{[i]}E,\ \ ET_c^{<i>}=T_{c-1}^{<i>}E,\label{Td12Eg-TE}\\
&&f_{\bm}E^j=\mbox{$\sum\limits_{i=0}^{j}$}\dbinom{j}{i}
\r_{i,\bm}^fE^{j-i}f_{\bm+i\bn},\label{Td12Eg-fE}\\
&&e_{\bm}E^j=\mbox{$\sum\limits_{i=0}^{j}$}\dbinom{j}{i}
\r_{i,\bm}^eE^{j-i}e_{\bm+i\bn},\label{Td12Eg-eE}\\
&&g_{\bm}E^j=\mbox{$\sum\limits_{i=0}^{j}$}\dbinom{j}{i}
\r_{i,\bm}^gE^{j-i}g_{\bm+i\bn},\label{Td12Eg-gE}\\
&&dE^j=E^jd,\ \ h_{\bm}E^j=E^jh_{\bm},\label{Td12Eg-dhE}\\
&&d_1E^j=E^jd_1+jn_1E^j,\ \ d_2E^j=E^jd_2+jn_2E^j.\label{Td12Eg-d1d2E}
\end{eqnarray}
\end{lemm}
\begin{proof}~Since $[T,l_{\bm}]=(x_1m_1+x_2m_2)l_{\bm}=rl_{\bm}$, $[d_1,E^j]=jn_1E^j$ and $[d_2,E^j]=jn_2E^j$, we obtain equations \eqref{Td12Eg-Tl} and \eqref{Td12Eg-d1d2E}. Equation \eqref{Td12Eg-TE} is a special case of \eqref{Td12Eg-Tl}. Equations \eqref{Td12Eg-dhE} are obtained by $[d,E]=[h,E]=0$.

Using induction on $i$, one has
\begin{eqnarray*}
&&(ad E)^ie_{\bm}=\mbox{$\prod\limits_{p=1}^i$} q^{n_2(m_1+(p-1)n_1)}e_{\bm+i\bn}=(-1)^i\r_{i,\bm}^ee_{\bm+i\bn},\\
&&(ad E)^if_{\bm}=(-1)^i\mbox{$\prod\limits_{p=1}^i$} q^{n_1(m_2+(p-1)n_2)}f_{\bm+i\bn}=(-1)^i\r_{i,\bm}^ff_{\bm+i\bn},\\
&&(ad E)^ig_{\bm}=\mbox{$\prod\limits_{p=1}^i$} (q^{n_2(m_1+(p-1)n_1)}-q^{n_1(m_2+(p-1)n_2)})g_{\bm+i\bn}=(-1)^i\r_{i,\bm}^gg_{\bm+i\bn}.
\end{eqnarray*}
Then, we obtain the equations \eqref{Td12Eg-fE}, \eqref{Td12Eg-eE}, \eqref{Td12Eg-gE} as follows:
\begin{eqnarray*}
&&f_{\bm}E^j=\mbox{$\sum\limits_{i=0}^{j}$}(-1)^i\dbinom{j}{i}E^{j-i}(ad E)^i(f_{\bm})
=\mbox{$\sum\limits_{i=0}^{j}\dbinom{j}{i}\r_{i,\bm}^f$}E^{j-i}f_{\bm+i\bn},\\
&&e_{\bm}E^j=\mbox{$\sum\limits_{i=0}^{j}$}(-1)^i\dbinom{j}{i}E^{j-i}(ad E)^i(e_{\bm})=\mbox{$\sum\limits_{i=0}^{j}$}\dbinom{j}{i}\r_{i,\bm}^eE^{j-i}e_{\bm+i\bn},\\
&&g_{\bm}E^j=\mbox{$\sum\limits_{i=0}^{j}$}(-1)^i\dbinom{j}{i}E^{j-i}(ad E)^i(g_{\bm})=\mbox{$\sum\limits_{i=0}^{j}$}\dbinom{j}{i}\r_{i,\bm}^gE^{j-i}g_{\bm+i\bn}.
\end{eqnarray*}
\end{proof}

\begin{lemm}\label{lemm-3-2}
The following identities hold in $\UU(\wsl)$ (where $l_{\bm}\in\WL_{\bm}$):
\begin{eqnarray}
&&(l_{\bm}\otimes1)I_c=I_{c-r}(l_{\bm}\otimes1),\label{Td12Eg-Il}\\
&&(1\otimes d_1)I_c=n_1I_{c+1}(T_{c}\otimes Et)+I_c(1\otimes d_1),\label{Td12Eg-Id1}\\
&&(1\otimes d_2)I_c=n_2I_{c+1}(T_{c}\otimes Et)+I_c(1\otimes d_2)\label{Td12Eg-Id2},\\
&&(1\otimes h_{\bm})I_c=I_c(1\otimes h_{\bm}),\ \ (1\otimes d)I_c=I_c(1\otimes d),\label{Td12Eg-Ihd}\\
&&(1\otimes f_{\bm})I_c=\mbox{$\sum\limits_{i=0}^{\infty}$}\frac{\r_{i,\bm}^f}{i!}I_{c+i}(T_c^{<i>}\otimes f_{\bm+i\bn}t^i),\label{Td12Eg-If}\\
&&(1\otimes e_{\bm})I_c=\mbox{$\sum\limits_{i=0}^{\infty}$}\frac{\r_{i,\bm}^e}{i!}I_{c+i}(T_c^{<i>}\otimes e_{\bm+i\bn}t^i),\label{Td12Eg-Ie}\\
&&(1\otimes g_{\bm})I_c=\mbox{$\sum\limits_{i=0}^{\infty}$}\frac{\r_{i,\bm}^g}{i!}I_{c+i}(T_c^{<i>}\otimes g_{\bm+i\bn}t^i).\label{Td12Eg-Ig}
\end{eqnarray}
\end{lemm}
\begin{proof}
For $l_{\bm}\in\WL_{\bm}$, using formula \eqref{Td12Eg-Tl},
\begin{eqnarray*}
(l_{\bm}\otimes1)I_c\!\!\!&=&\!\!\!(l_{\bm}\otimes1)(\mbox{$\sum\limits_{i=0}^{\infty}$}\frac{1}{i!}T_c^{<i>}\otimes E^it^i)\\
\!\!\!&=&\!\!\!\mbox{$\sum\limits_{i=0}^{\infty}$}\frac{1}{i!}T_{c-r}^{<i>}l_{\bm}\otimes E^it^i=I_{c-r}(l_{\bm}\otimes1).
\end{eqnarray*}
Hence, we obtain equation \eqref{Td12Eg-Il}. It is obvious that $d_1E^i=[d_1,E^i]+E^id_1=in_1E^i+E^id_1$ and $d_2E^i=[d_2,E^i]+E^id_2=in_2E^i+E^id_2$, which mean
\begin{eqnarray*}
(1\otimes d_1)I_c\!\!\!\!\!\!\!&&=\mbox{$\sum\limits_{i=0}^{\infty}$}\frac{1}{i!}T_c^{<i>}\otimes d_1E^it^i
=\mbox{$\sum\limits_{i=0}^{\infty}$}\frac{1}{i!}T_c^{<i>}\otimes(in_1E^i+E^id_1)t^i\\
&&=\mbox{$\sum\limits_{i=0}^{\infty}$}\frac{n_1}{i!}T_c^{<i+1>}\otimes E^{i+1}t^{i+1}+\mbox{$\sum\limits_{i=0}^{\infty}$}\frac{1}{i!}T_c^{<i>}\otimes E^id_1t^i\\
&&=\mbox{$\sum\limits_{i=0}^{\infty}$}\frac{n_1}{i!}T_cT_{c+1}^{<i>}\otimes E^{i+1}t^{i+1}+I_c(1\otimes d_1)\\
&&=n_1I_{c+1}(T_c\otimes Et)+I_c(1\otimes d_1),\\
(1\otimes d_2)I_c\!\!\!\!\!\!\!&&=\mbox{$\sum\limits_{i=0}^{\infty}$}\frac{1}{i!}T_c^{<i>}\otimes d_2E^it^i
=\mbox{$\sum\limits_{i=0}^{\infty}$}\frac{1}{i!}T_c^{<i>}\otimes(in_2E^i+E^id_2)t^i\\
&&=\mbox{$\sum\limits_{i=0}^{\infty}$}\frac{n_2}{i!}T_c^{<i+1>}\otimes E^{i+1}t^{i+1}+\mbox{$\sum\limits_{i=0}^{\infty}$}\frac{1}{i!}T_c^{<i>}\otimes E^id_2t^i\\
&&=\mbox{$\sum\limits_{i=0}^{\infty}$}\frac{n_2}{i!}T_cT_{c+1}^{<i>}\otimes E^{i+1}t^{i+1}+I_c(1\otimes d_2)\\
&&=n_2I_{c+1}(T_c\otimes Et)+I_c(1\otimes d_2).
\end{eqnarray*}
Hence, we complete the proof of equations \eqref{Td12Eg-Id1}, \eqref{Td12Eg-Id2} respectively.

Since $[h_{\bm},E]=[d,E]=0$, for all $h_{\bm}\in\HH_{\bm}$, equation \eqref{Td12Eg-Ihd} is obviously established.

Using formula \eqref{x<>[]}, \eqref{Td12Eg-fE}, we obtain equations \eqref{Td12Eg-If} as follows,
\begin{eqnarray*}
(1\otimes f_{\bm})I_c\!\!\!\!\!\!\!&&=\mbox{$\sum\limits_{i=0}^{\infty}$}\frac{1}{i!}T_c^{<i>}\otimes f_{\bm}E^it^i\\
&&=\mbox{$\sum\limits_{i=0}^{\infty}$}\frac{1}{i!}T_c^{<i>}\otimes
\big(\mbox{$\sum\limits_{j=0}^i$}
\dbinom{i}{j}\r_{j,\bm}^fE^{i-j}f_{\bm+j\bn}\big)t^i\\
&&=\mbox{$\sum\limits_{i=0}^{\infty}$}\mbox{$\sum\limits_{j=0}^{\infty}$}
\frac{1}{(i+j)!}\dbinom{i+j}{j}
\r_{j,\bm}^fT_c^{<i+j>}\otimes E^{i}f_{\bm+j\bn}t^{i+j}\\
&&=\mbox{$\sum\limits_{j=0}^{\infty}$}\frac{1}{j!}\r_{j,\bm}^fT_c^{<j>}
\mbox{$\sum\limits_{i=0}^{\infty}$}
\frac{1}{i!}T_{c+j}^{<i>}\otimes E^if_{\bm+j\bn}t^{i+j}\\
&&=\mbox{$\sum\limits_{j=0}^{\infty}$}\frac{1}{j!}\r_{j,\bm}^f
I_{c+j}(T_{c}^{<j>}\otimes f_{\bm+j\bn}t^j).
\end{eqnarray*}
Similarly, equations \eqref{Td12Eg-Ie} and \eqref{Td12Eg-Ig} are also tenable. Now, we complete the proof of this lemma.
\end{proof}

\begin{lemm}\label{lemm-3-3}
The following identities hold in $\UU(\wsl)$:
\begin{eqnarray*}
&&h_{\bm}J_c=J_{c+r}h_{\bm},\ \ dJ_c=J_{c}d,\\
&&d_1J_c=J_cd_1-n_1J_cT_{-c}Et,\\
&&d_2J_c=J_cd_2-n_2J_cT_{-c}Et,\\
&&f_{\bm}J_c=J_{c+r}(\mbox{$\sum\limits_{j=0}^{\infty}$}
\frac{(-1)^j\r_{j,\bm}^f}{j!}f_{\bm+j\bn}T_{1-c}^{<j>}t^j),\\
&&e_{\bm}J_c=J_{c+r}(\mbox{$\sum\limits_{j=0}^{\infty}$}
\frac{(-1)^j\r_{j,\bm}^e}{j!}e_{\bm+j\bn}T_{1-c}^{<j>}t^j),\\
&&g_{\bm}J_c=J_{c+r}(\mbox{$\sum\limits_{j=0}^{\infty}$}
\frac{(-1)^j\r_{j,\bm}^g}{j!}g_{\bm+j\bn}T_{1-c}^{<j>}t^j).
\end{eqnarray*}
\end{lemm}
\begin{proof}
Using the formulae \eqref{x<>[]}, \eqref{equa-J}, \eqref{Td12Eg-Tl} and \eqref{Td12Eg-dhE}, we have
\begin{eqnarray*}
h_{\bm}J_c=\mbox{$\sum\limits_{i=0}^{\infty}$}\frac{(-1)^i}{i!}h_{\bm}T_{-c}^{[i]}E^it^i
=\mbox{$\sum\limits_{i=0}^{\infty}$}\frac{(-1)^i}{i!}T_{-c-r}^{[i]}h_{\bm}E^it^i=J_{c+r}h_{\bm}.
\end{eqnarray*}
Similarly, $dJ_c=J_cd$.
Since formulas \eqref{x<>[]}, \eqref{equa-J}, \eqref{Td12Eg-Tl} and \eqref{Td12Eg-d1d2E}, there are
\begin{eqnarray*}
d_1J_c\!\!\!\!\!\!\!&&=\mbox{$\sum\limits_{i=0}^{\infty}$}\frac{(-1)^i}{i!}d_1T_{-c}^{[i]}E^it^i
=\mbox{$\sum\limits_{i=0}^{\infty}$}\frac{(-1)^i}{i!}T_{-c}^{[i]}d_1E^it^i
=\mbox{$\sum\limits_{i=0}^{\infty}$}\frac{(-1)^i}{i!}T_{-c}^{[i]}(in_1E^i+E^id_1)t^i\\
&&=\mbox{$\sum\limits_{i=0}^{\infty}$}\frac{(-1)^i}{i!}T_{-c}^{[i]}E^id_1t^i
+\mbox{$\sum\limits_{i=0}^{\infty}$}\frac{(-1)^{i+1}n_1}{i!}T_{-c}^{[i]}T_{-c-i}E^{i+1}t^{i+1}
=J_{c}d_1-n_1J_{c}T_{-c}Et,\\
d_2J_c\!\!\!\!\!\!\!&&=\mbox{$\sum\limits_{i=0}^{\infty}$}\frac{(-1)^i}{i!}d_2T_{-c}^{[i]}E^it^i
=\mbox{$\sum\limits_{i=0}^{\infty}$}\frac{(-1)^i}{i!}T_{-c}^{[i]}d_2E^it^i
=\mbox{$\sum\limits_{i=0}^{\infty}$}\frac{(-1)^i}{i!}T_{-c}^{[i]}(in_2E^i+E^id_2)t^i\\
&&=\mbox{$\sum\limits_{i=0}^{\infty}$}\frac{(-1)^i}{i!}T_{-c}^{[i]}E^id_2t^i
+\mbox{$\sum\limits_{i=0}^{\infty}$}\frac{(-1)^{i+1}n_2}{i!}T_{-c}^{[i]}T_{-c-i}E^{i+1}t^{i+1}
=J_{c}d_2-n_2J_{c}T_{-c}Et.
\end{eqnarray*}

The last three equations could be obtained by the formulas \eqref{x<>[]}, \eqref{equa-J} and formulas from \eqref{Td12Eg-Tl} to \eqref{Td12Eg-gE}. For symbol $y_{\bm}=f_{\bm},e_{\bm}$ or $g_{\bm}$, there is
\begin{eqnarray*}
y_{\bm}J_c\!\!\!\!\!\!\!&&=\mbox{$\sum\limits_{i=0}^{\infty}$}
\frac{(-1)^i}{i!}y_{\bm}T_{-c}^{[i]}E^it^i
=\mbox{$\sum\limits_{i=0}^{\infty}$}\frac{(-1)^i}{i!}
T_{-c-r}^{[i]}y_{\bm}E^it^i\\
&&=\mbox{$\sum\limits_{i=0}^{\infty}$}\frac{(-1)^i}{i!}
T_{-c-r}^{[i]}(\mbox{$\sum\limits_{j=0}^{i}$}\dbinom{i}{j}
\r_{j,\bm}^yE^{i-j}y_{\bm+j\bn})t^i\\
&&=\mbox{$\sum\limits_{i=0}^{\infty}$}\mbox{$\sum\limits_{j=0}^{\infty}$}
\frac{(-1)^{i+j}}{i!j!}\r_{j,\bm}^y
T_{-c-r}^{[i]}T_{-c-r-i}^{[j]}E^{i}y_{\bm+j\bn}t^{i+j}\\
&&=\mbox{$\sum\limits_{i=0}^{\infty}$}\frac{(-1)^i}{i!}T_{-c-r}^{[i]}E^it^i
\mbox{$\sum\limits_{j=0}^{\infty}$}\frac{(-1)^j}{j!}
\r_{j,\bm}^yy_{\bm+j\bn}T_{-c+j}^{[j]}t^j\\
&&=J_{c+r}\mbox{$\sum\limits_{j=0}^{\infty}$}\frac{(-1)^j}{j!}
\r_{j,\bm}^yy_{\bm+j\bn}T_{1-c}^{<j>}t^j.
\end{eqnarray*}
\end{proof}

\begin{proof}[Proof of Theorem \ref{main}\,(1)]
Using equations \eqref{equa-D-D0}, \eqref{equa-II-JJ} and all the
lemmas above, for symbol $y_{\bm}=f_{\bm},e_{\bm}$ or $g_{\bm}$, we obtain
\begin{eqnarray*}
&\D(y_{\bm})\!\!\!\!&=\II \D_0(y_{\bm}) \II^{-1}=\II(y_{\bm}\otimes1+1\otimes y_{\bm})I\\
&&=\II I_{-r}(y_{\bm}\otimes1)+\II(\mbox{$\sum\limits_{i=0}^{\infty}$}
\frac{\r_{i,\bm}^y}{i!}I_{i}(T^{<i>}\otimes y_{\bm+i\bn}t^i))\\
&&=(1\otimes(1-Et)^r)(y_{\bm}\otimes1)+
\mbox{$\sum\limits_{i=0}^{\infty}$}\frac{\r_{i,\bm}^y}{i!}
(1\otimes(1-Et)^{-i})(T^{<i>}\otimes y_{\bm+i\bn}t^i)\\
&&=y_{\bm}\otimes(1-Et)^r+\mbox{$\sum\limits_{i=0}^{\infty}$}
\frac{\r_{i,\bm}^y}{i!}
T^{<i>}\otimes(1-Et)^{-i} y_{\bm+i\bn}t^i,\\
&\D(h_{\bm})\!\!\!\!&=\II \D_0(h_{\bm})
\II^{-1}=\II(h_{\bm}\otimes1+1\otimes h_{\bm})I\\
&&=\II I_{-r}(h_{\bm}\otimes1)+\II I (1\otimes h_{\bm})=h_{\bm}\otimes(1-Et)^{r}+1\otimes h_{\bm},\\
&\!\!\!\D(d)\!\!\!\!\!\!&=\II \D_0(d) \II^{-1}=\II(d\otimes1+1\otimes d)I=d\otimes1+1\otimes d,\\
&\D(d_1)\!\!\!\!&=\II \D_0(d_1) \II^{-1}=\II(d_1\otimes1+1\otimes d_1)I\\
&&=\II I(d_1\otimes1)+\II (n_1I_{1}(T\otimes Et)+I(1\otimes d_1))\\
&&=d_1\otimes1+1\otimes d_1+n_1T\otimes(1-Et)^{-1}Et\\
&&=d_1\otimes1+1\otimes d_1-n_1T\otimes1+n_1T\otimes(1-Et)^{-1},\\
&\D(d_2)\!\!\!\!&=\II \D_0(d_2) \II^{-1}=\II(d_2\otimes1+1\otimes d_2)I\\
&&=\II I(d_2\otimes1)+\II (n_2I_{1}(T\otimes Et)+I(1\otimes d_2))\\
&&=d_2\otimes1+1\otimes d_2+n_2T\otimes(1-Et)^{-1}Et\\
&&=d_2\otimes1+1\otimes d_2-n_2T\otimes1+n_2T\otimes(1-Et)^{-1}.
\end{eqnarray*}

In addition, we also obtain,
\begin{eqnarray*}
&S(y_{\bm})\!\!\!\!&=\JJ s_0(y_{\bm}) J=-\JJ y_{\bm}J=-\JJ(J_{r}\mbox{$\sum\limits_{j=0}^{\infty}$}
\frac{(-1)^j\r_{j,\bm}^y}{j!}y_{\bm+j\bn}T_{1-c}^{<j>}t^j)\\
&&=\mbox{$\sum\limits_{j=0}^{\infty}$}\frac{(-1)^{j+1}
\r_{j,\bm}^y}{j!}(1-Et)^{-r}y_{\bm+j\bn}T_{1-c}^{<j>}t^j,\\
&S(h_{\bm})\!\!\!\!&=\JJ s_0(h_{\bm}) J=-\JJ h_{\bm}J=-\JJ J_r h_{\bm}=-(1-Et)^{-r}h_{\bm},\\
&S(d)\!\!\!\!&=\JJ s_0(d) J=-\JJ d J=-d,\\
&S(d_1)\!\!\!\!&=\JJ s_0(d_1) J=-\JJ d_1J=-\JJ(Jd_1-n_1JTEt)=-d_1+n_1TEt,\\
&S(d_2)\!\!\!\!&=\JJ s_0(d_2) J=-\JJ d_2 J=-\JJ(Jd_2-n_2JTEt)=-d_2+n_2TEt.
\end{eqnarray*}
\end{proof}

\vskip12pt

\noindent{\bf4. Proof of Theorem \ref{main}(2)}\setcounter{section}{4}\setcounter{theo}{0}\setcounter{equation}{0}

\vskip6pt

In this section, we take $T=x_1d_1+x_2d_2$ for some $x_1,x_2\in\C$, $\bn=(n_1,n_2)\in\BZ$ and $E=e_{\bn}$ such that $[T,E]=E$. It is easy to see that $x_1n_1+x_2n_2=1$. The expressions only referring to $T$ in Section 3 are also tenable in this section, such as expressions \eqref{Td12Eg-Tl}, \eqref{Td12Eg-TE}, \eqref{Td12Eg-Il}.
For $\bm=(m_1,m_2)\in\BZ$, $r=x_1m_1+x_2m_2$, denote
\begin{eqnarray*}
s_{\bm}=q^{m_1n_2+m_2n_1+n_1n_2}.
\end{eqnarray*}

\begin{lemm}\label{lemm-4-1}
The following identities hold in $\UU(\wsl)$:
\begin{eqnarray}
&&e_{\bm}E^j=E^je_{\bm},\ \ dE^j=E^jd+2jE^j,\label{Td12Ee-deE}\\
&&d_1E^j=E^jd_1+jn_1E^j,\ \ d_2E^j=E^jd_2+jn_2E^j,\label{Td12Ee-d1d2E}\\
&&f_{\bm}E^j=\begin{cases}
jq^{m_2n_1}E^{j-1}h_{\bm+\bn}-jq^{m_1n_2}E^{j-1}g_{\bm+\bn}\\
+E^jf_{\bm}-2s_{\bm}\dbinom{j}{2}E^{j-2}e_{\bm+2\bn},&\bm+\bn\neq0,\\
E^jf_{-\bn}-jq^{-n_1n_2}E^{j-1}d-2\dbinom{j}{2}q^{-n_2n_1}E^{j-1},&\bm+\bn=0,
\end{cases}\label{Td12Ee-fE}\\
&&g_{\bm}E^j=E^jg_{\bm}+jq^{m_2n_1}E^{j-1}e_{\bm+\bn},\label{Td12Ee-gE}\\
&&h_{\bm}E^j=E^jh_{\bm}-jq^{m_1n_2}E^{j-1}e_{\bm+\bn}.\label{Td12Ee-hE}
\end{eqnarray}
\end{lemm}
\begin{proof}~Equations \eqref{Td12Ee-deE} and \eqref{Td12Ee-d1d2E} are obtained by formulas $[e_{\bm},E]=0$, $[d,E]=2E$ and $[d_1,E^j]=jn_1E^j$, $[d_2,E^j]=jn_2E^j$ respectively .

By the definition, if $\bm+\bn\neq0$, there is
\begin{eqnarray*}
&&(ad E)^i(f_{\bm})=(ad e_{\bn})^i(f_{\bm})=\begin{cases}
f_{\bm},&i=0,\\
q^{m_1n_2}g_{\bm+\bn}-q^{m_2n_1}h_{\bm+\bn},&i=1,\\
-2s_{\bm}e_{\bm+2\bn},&i=2,\\
0,&i>2,
\end{cases}
\end{eqnarray*}
where $s_{\bm}=q^{m_1n_2+m_2n_1+n_1n_2}\in\C$.
Thus,
\begin{eqnarray*}
f_{\bm}E^j\!\!\!\!\!\!\!\!\!&&=\mbox{$\sum\limits_{i=0}^{j}$}(-1)^i\dbinom{j}{i}E^{j-i}(ad E)^i(f_{\bm})\\
&&=E^jf_{\bm}-jE^{j-1}(q^{m_1n_2}g_{\bm+\bn}-q^{m_2n_1}h_{\bm+\bn})
+\dbinom{j}{2}E^{j-2}(-2s_{\bm}e_{\bm+2\bn})\\
&&=E^jf_{\bm}-jq^{m_1n_2}E^{j-1}g_{\bm+\bn}+jq^{m_2n_1}E^{j-1}h_{\bm+\bn}
-2s_{\bm}\dbinom{j}{2}E^{j-2}e_{\bm+2\bn}.
\end{eqnarray*}
Similarly, there are
\begin{eqnarray*}
&&(ad E)^i(f_{-\bn})=(ad e_{\bn})^i(f_{-\bn})=\begin{cases}
f_{-\bn},&i=0,\\
q^{-n_1n_2}d,&i=1,\\
-2q^{-n_1n_2}e_{\bn},&i=2,\\
0,&i>2,
\end{cases}
\end{eqnarray*}
and
\begin{eqnarray*}
f_{-\bn}E^j\!\!\!\!\!\!&&=\mbox{$\sum\limits_{i=0}^{j}$}(-1)^i\dbinom{j}{i}E^{j-i}(ad E)^i(f_{-\bn})\\
&&=E^jf_{-\bn}-jE^{j-1}(q^{-n_1n_2}d)+\dbinom{j}{2}E^{j-2}(-2q^{-n_1n_2}E)\\
&&=E^jf_{-\bn}-jq^{-n_1n_2}E^{j-1}d-2q^{-n_1n_2}\dbinom{j}{2}E^{j-1}.
\end{eqnarray*}
Furthermore, one can obtain
\begin{eqnarray*}
&&(ad E)^i(g_{\bm})=(ad e_{\bn})^i(g_{\bm})=\begin{cases}
g_{\bm},&i=0,\\
-q^{m_2n_1
}e_{\bm+\bn},&i=1,\\
0,&i\geq2.
\end{cases}\\
&&(ad E)^i(h_{\bm})=(ad e_{\bn})^i(h_{\bm})=\begin{cases}
h_{\bm},&i=0,\\
q^{m_1n_2}e_{\bm+\bn},&i=1,\\
0,&i\geq2.
\end{cases}\\
\end{eqnarray*}
These mean
\begin{eqnarray*}
&&g_{\bm}E^j=\mbox{$\sum\limits_{i=0}^{j}$}(-1)^i\dbinom{j}{i}E^{j-i}(ad E)^i(g_{\bm})=E^jg_{\bm}+jq^{m_2n_1}E^{j-1}e_{\bm+\bn},\\
&&h_{\bm}E^j=\mbox{$\sum\limits_{i=0}^{j}$}(-1)^i\dbinom{j}{i}E^{j-i}(ad E)^i(h_{\bm})=E^jh_{\bm}-jq^{m_1n_2}E^{j-1}e_{\bm+\bn}.\\
\end{eqnarray*}
\end{proof}

\begin{lemm}\label{lemm-4-2}
The following identities hold in $\UU(\wsl)$ (where $l_{\bm}\in\WL_{\bm}$):
\begin{eqnarray}
(l_{\bm}\otimes1)I_c\!\!\!\!\!\!\!\!&&=I_{c-r}(l_{\bm}\otimes1),\label{Td12Ee-Il}\\
(1\otimes d_1)I_c\!\!\!\!\!\!\!\!&&=n_1I_{c+1}(T_{c}\otimes Et)+I_c(1\otimes d_1),\label{Td12Ee-Id1}\\
(1\otimes d_2)I_c\!\!\!\!\!\!\!\!&&=n_2I_{c+1}(T_{c}\otimes Et)+I_c(1\otimes d_2)\label{Td12Ee-Id2},\\
(1\otimes e_{\bm})I_c\!\!\!\!\!\!\!\!&&=I_{c}(1\otimes e_{\bm}),\label{Td12Ee-Ie}\\
(1\otimes f_{\bm})I_c\!\!\!\!\!\!\!\!&&=\begin{cases}
q^{m_2n_1}I_{c+1}(T_c\otimes h_{\bm+\bn}t)-q^{m_1n_2}I_{c+1}(T_c\otimes g_{\bm+\bn}t)\\
+I_c(1\otimes f_{\bm})-s_{\bm}I_{c+2}(T_c^{<2>}\otimes e_{\bm+2\bn}t^2),&\bm+\bn\neq0,\\ \\
I_c(1\otimes f_{-\bn})-q^{-n_2n_1}I_{c+1}(T_c\otimes dt)\\
-q^{-n_1n_2}I_{c+2}(T_c^{<2>}\otimes Et^2),&\bm+\bn=0,
\end{cases}\label{Td12Ee-If}\\
(1\otimes g_{\bm})I_c\!\!\!\!\!\!\!\!&&=I_c(1\otimes g_{\bm})+q^{m_2n_1}I_{c+1}(T_c\otimes e_{\bm+\bn}t),\label{Td12Ee-Ig}\\
(1\otimes h_{\bm})I_c\!\!\!\!\!\!\!\!&&=I_c(1\otimes h_{\bm})-q^{m_1n_2}I_{c+1}(T_c\otimes e_{\bm+\bn}t),\label{Td12Ee-Ih}\\
(1\otimes d)I_c\!\!\!\!\!\!\!\!&&=I_c(1\otimes d)+2I_{c+1}(T_c\otimes Et
).\label{Td12Eg-Id}
\end{eqnarray}
\end{lemm}
\begin{proof}
Equations \eqref{Td12Ee-Il}, \eqref{Td12Ee-Id1} and \eqref{Td12Ee-Id2} are similar as \eqref{Td12Eg-Il}, \eqref{Td12Eg-Id1} and \eqref{Td12Eg-Id2} in Lemma \ref{lemm-3-2}. It is easy to obtain equation \eqref{Td12Ee-Ie} by $[e_{\bm},E]=0$.

Using formula \eqref{x<>[]}, \eqref{Td12Ee-fE}, we obtain equation \eqref{Td12Ee-If} as follows, for $\bm+\bn\neq0$,
\begin{eqnarray*}
(1\otimes f_{\bm})I_c\!\!\!\!\!\!\!\!\!&&=\mbox{$\sum\limits_{i=0}^{\infty}$}\frac{1}{i!}T_c^{<i>}\otimes f_{\bm}E^it^i=\mbox{$\sum\limits_{i=0}^{\infty}$}\frac{1}{i!}T_c^{<i>}\otimes\\
&&\ \ \ \ \big(E^if_{\bm}\!-iq^{n_2m_1}E^{i-1}g_{\bm+\bn}
\!+iq^{n_1m_2}E^{i-1}h_{\bm+\bn} \!-2s_{\bm}\dbinom{i}{2}E^{i-2}e_{\bm+2\bn}\big)t^i\\
&&=\mbox{$\sum\limits_{i=0}^{\infty}$}\frac{1}{i!}T_c^{<i>}\otimes E^if_{\bm}t^i-q^{n_2m_1}\mbox{$\sum\limits_{i=0}^{\infty}$}\frac{1}{i!}T_c^{<i+1>}\otimes E^ig_{\bm+\bn}t^{i+1}\\
&&\ \ \ \ +q^{n_1m_2}\mbox{$\sum\limits_{i=0}^{\infty}$}\frac{1}{i!}T_c^{<i+1>}\otimes E^ih_{\bm+\bn}t^{i+1}-s_{\bm}
\mbox{$\sum\limits_{i=0}^{\infty}$}\frac{1}{i!}T_c^{<i+2>}\otimes E^ie_{\bm+2\bn}t^{i+2}\\
&&=I_c(1\otimes f_{\bm})-q^{n_2m_1}I_{c+1}(T_c\otimes g_{\bm+\bn}t)+q^{n_1m_2}I_{c+1}(T_c\otimes h_{\bm+\bn}t)\\
&&\ \ \ \ -s_{\bm}
I_{c+2}(T_c^{<2>}\otimes e_{\bm+2\bn}t^2).
\end{eqnarray*}
Similarly
\begin{eqnarray*}
(1\otimes f_{-\bn})I_c\!\!\!\!\!\!\!\!\!&&=\mbox{$\sum\limits_{i=0}^{\infty}$}\frac{1}{i!}T_c^{<i>}\otimes f_{-\bn}E^it^i\\
&&=\mbox{$\sum\limits_{i=0}^{\infty}$}\frac{1}{i!}T_c^{<i>}\otimes\big(E^if_{-\bn}-iq^{-n_2n_1}E^{i-1}d
-2q^{-n_2n_1}\dbinom{i}{2}E^{i-1}\big)t^i\\
&&=\mbox{$\sum\limits_{i=0}^{\infty}$}\frac{1}{i!}T_c^{<i>}\otimes E^if_{-\bn}t^i-q^{-n_2n_1}\mbox{$\sum\limits_{i=0}^{\infty}$}\frac{1}{i!}T_c^{<i+1>}\otimes E^idt^{i+1}\\
&&\ \ \ \ -q^{-n_1n_2}\mbox{$\sum\limits_{i=0}^{\infty}$}\frac{1}{i!}T_c^{<i+2>}\otimes E^{i+1}t^{i+2}\\
&&=I_c(1\otimes f_{-\bn})-q^{-n_2n_1}I_{c+1}(T_c\otimes dt)-q^{-n_1n_2}I_{c+2}(T_c^{<2>}\otimes Et^2).
\end{eqnarray*}
Moreover, we also have
\begin{eqnarray*}
(1\otimes g_{\bm})I_c\!\!\!\!\!\!&&=\mbox{$\sum\limits_{i=0}^{\infty}$}\frac{1}{i!}T_c^{<i>}\otimes g_{\bm}E^it^i=\mbox{$\sum\limits_{i=0}^{\infty}$}\frac{1}{i!}T_c^{<i>}\otimes (E^ig_{\bm}+iq^{m_2n_1}E^{i-1}e_{\bm+\bn})t^i\\
&&=\mbox{$\sum\limits_{i=0}^{\infty}$}\frac{1}{i!}T_c^{<i>}\otimes E^ig_{\bm}t^i+q^{m_2n_1}\mbox{$\sum\limits_{i=0}^{\infty}$}\frac{1}{i!}T_c^{<i+1>}\otimes E^ie_{\bm+\bn}t^{i+1}\\
&&=I_c(1\otimes g_{\bm})+q^{m_2n_1}I_{c+1}(T_c\otimes e_{\bm+\bn}t),\\
(1\otimes h_{\bm})I_c\!\!\!\!\!\!\!\!\!&&=\mbox{$\sum\limits_{i=0}^{\infty}$}\frac{1}{i!}T_c^{<i>}\otimes h_{\bm}E^it^i=\mbox{$\sum\limits_{i=0}^{\infty}$}\frac{1}{i!}T_c^{<i>}\otimes (E^ih_{\bm}-iq^{m_1n_2}E^{i-1}e_{\bm+\bn})t^i\\
&&=\mbox{$\sum\limits_{i=0}^{\infty}$}\frac{1}{i!}T_c^{<i>}\otimes E^ih_{\bm}t^i-q^{m_1n_2}\mbox{$\sum\limits_{i=0}^{\infty}$}\frac{1}{i!}T_c^{<i+1>}\otimes E^ie_{\bm+\bn}t^{i+1}\\
&&=I_c(1\otimes h_{\bm})-q^{m_1n_2}I_{c+1}(T_c\otimes e_{\bm+\bn}t),\\
(1\otimes d)I_c\!\!\!\!\!\!\!\!\!&&=\mbox{$\sum\limits_{i=0}^{\infty}$}\frac{1}{i!}T_c^{<i>}\otimes dE^it^i=\mbox{$\sum\limits_{i=0}^{\infty}$}\frac{1}{i!}T_c^{<i>}\otimes (E^id+2iE^i)t^i\\
&&=I_c(1\otimes d)+2I_{c+1}(T_c\otimes Et).
\end{eqnarray*}\end{proof}

\begin{lemm}\label{lemm-4-3}
The following identities hold in $\UU(\wsl)$:
\begin{eqnarray}
&&e_{\bm}J_c=J_{c+r}e_{\bm},\label{Td12Ee-ej}\\
&&d_1J_c=J_cd_1-n_1J_cT_{-c}Et,\label{Td12Ee-d1j}\\
&&d_2J_c=J_cd_2-n_2J_cT_{-c}Et,\label{Td12Ee-d2j}\\
&&f_{\bm}J_c=\begin{cases}
q^{m_1n_2}J_{c+r}(T_{-c-r}g_{\bm+\bn}t)-q^{m_2n_1}J_{c+r}(h_{\bm+\bn}T_{1-c}t)\\
+J_{c+r}f_{\bm}-s_{\bm}J_{c+r}(e_{\bm+2\bn}T_{1-c}^{<2>}t^2),&\bm+\bn\neq0,\\ \\
J_{c-1}f_{-\bn}+q^{n_1n_2}J_{c-1}(dT_{1-c}t)-q^{n_1n_2}J_{c-1}ET_{1-c}^{<2>}t^2,&\bm+\bn=0,
\end{cases}\label{Td12Ee-fj}\\
&&g_{\bm}J_c=J_{c+r}g_{\bm}-q^{n_1m_2}J_{c+r}(e_{\bm+\bn}T_{1-c}t),\label{Td12Ee-gj}\\
&&h_{\bm}J_c=J_{c+r}h_{\bm}+q^{n_2m_1}J_{c+r}(e_{\bm+\bn}T_{1-c}t),\label{Td12Ee-hj}\\
&&dJ_c=J_cd-2J_c(ET_{1-c}t).\label{Td12Ee-dj}
\end{eqnarray}
\begin{proof}
Using the formulas \eqref{equa-J}, \eqref{Td12Eg-Tl} and \eqref{Td12Ee-deE}, there is
\begin{eqnarray*}
e_{\bm}J_c\!\!\!\!\!\!\!\!\!&&=\mbox{$\sum\limits_{i=0}^{\infty}$}\frac{(-1)^i}{i!}e_{\bm}T_{-c}^{[i]}E^it^i
=\mbox{$\sum\limits_{i=0}^{\infty}$}\frac{(-1)^i}{i!}T_{-c-r}^{[i]}e_{\bm}E^it^i\\
&&=\mbox{$\sum\limits_{i=0}^{\infty}$}\frac{(-1)^i}{i!}T_{-c-r}^{[i]}E^ie_{\bm}t^i=J_{c+r}e_{\bm}.
\end{eqnarray*}

Formulas \eqref{Td12Ee-d1j} and \eqref{Td12Ee-d2j} are the same as those presented in Lemma \ref{lemm-3-3}. For $\bm\neq-\bn$, there is
\begin{eqnarray*}
f_{\bm}J_c\!\!\!\!\!\!\!&&=\mbox{$\sum\limits_{i=0}^{\infty}$}\frac{(-1)^i}{i!}f_{\bm}T_{-c}^{[i]}E^it^i
=\mbox{$\sum\limits_{i=0}^{\infty}$}\frac{(-1)^i}{i!}T_{-c-r}^{[i]}f_{\bm}E^it^i\\
&&=\mbox{$\sum\limits_{i=0}^{\infty}$}\frac{(-1)^i}{i!}T_{-c-r}^{[i]}(E^if_{\bm}\!\!-iq^{m_1n_2}E^{i-1}g_{\bm+\bn}
\!\!+iq^{m_2n_1}E^{i-1}h_{\bm+\bn}\!\!-2s_{\bm}\dbinom{i}{2}E^{i-2}e_{\bm+2\bn})t^i\\
&&=\mbox{$\sum\limits_{i=0}^{\infty}$}\frac{(-1)^i}{i!}T_{-c-r}^{[i]}E^if_{\bm}t^i-q^{m_1n_2}\mbox{$\sum\limits_{i=0}^{\infty}$}\frac{(-1)^{i+1}}{i!}T_{-c-r}^{[i+1]}E^{i}g_{\bm+\bn}t^{i+1}\\
&&\ \ \ \ +q^{m_2n_1}\mbox{$\sum\limits_{i=0}^{\infty}$}\frac{(-1)^{i+1}}{i!}T_{-c-r}^{[i+1]}E^{i}h_{\bm+\bn}t^{i+1}-s_{\bm}\mbox{$\sum\limits_{i=0}^{\infty}$}\frac{(-1)^{i+2}}{i!}T_{-c-r}^{[i+2]}E^{i}e_{\bm+2\bn})t^{i+2}\\
&&=J_{c+r}f_{\bm}+q^{m_1n_2}J_{c+r}(g_{\bm+\bn}T_{1-c}t)-q^{m_2n_1}J_{c+r}(h_{\bm+\bn}T_{1-c}t)-s_{\bm}J_{c+r}(e_{\bm+2\bn}T_{1-c}^{<2>}t^2).
\end{eqnarray*}

Furthermore, there is
\begin{eqnarray*}
f_{-\bn}J_c\!\!\!\!\!\!\!&&=\mbox{$\sum\limits_{i=0}^{\infty}$}\frac{(-1)^i}{i!}f_{-\bn}T_{-c}^{[i]}E^it^i
=\mbox{$\sum\limits_{i=0}^{\infty}$}\frac{(-1)^i}{i!}T_{-c+1}^{[i]}f_{-\bn}E^it^i\\
&&=\mbox{$\sum\limits_{i=0}^{\infty}$}\frac{(-1)^i}{i!}T_{-c+1}^{[i]}(E^if_{-\bn}-iq^{-n_2n_1}E^{i-1}d
-2q^{-n_2n_1}\dbinom{i}{2}E^{i-1})t^i\\
&&=\mbox{$\sum\limits_{i=0}^{\infty}$}\frac{(-1)^i}{i!}T_{-c+1}^{[i]}E^if_{-\bn}t^i
-q^{-n_2n_1}\mbox{$\sum\limits_{i=0}^{\infty}$}\frac{(-1)^{i+1}}{i!}T_{-c+1}^{[i+1]}E^{i}dt^{i+1}\\
&&\ \ \ \ -q^{-n_1n_2}\mbox{$\sum\limits_{i=0}^{\infty}$}\frac{(-1)^{i+2}}{i!}T_{-c+1}^{[i+2]}E^{i+1}t^{i+2}\\
&&=J_{c-1}f_{-\bn}+q^{-n_2n_1}J_{c-1}T_{1-c}dt-q^{-n_1n_2}J_{c-1}T_{1-c}^{[2]}Et^2\\
&&=J_{c-1}f_{-\bn}+q^{-n_1n_2}J_{c-1}dT_{1-c}t-q^{-n_1n_2}J_{c-1}ET_{1-c}^{<2>}t^2.
\end{eqnarray*}
Thus, we obtain \eqref{Td12Ee-fj}. Equations \eqref{Td12Ee-gj}  to \eqref{Td12Ee-dj} could be obtained by the formulas \eqref{x<>[]}, \eqref{equa-J} and formulas \eqref{Td12Ee-deE}, \eqref{Td12Ee-gE} and \eqref{Td12Ee-hE},
\begin{eqnarray*}
g_{\bm}J_c\!\!\!\!\!\!\!&&=\mbox{$\sum\limits_{i=0}^{\infty}$}\frac{(-1)^i}{i!}g_{\bm}T_{-c}^{[i]}E^it^i
=\mbox{$\sum\limits_{i=0}^{\infty}$}\frac{(-1)^i}{i!}T_{-c-r}^{[i]}g_{\bm}E^it^i\\
&&=\mbox{$\sum\limits_{i=0}^{\infty}$}\frac{(-1)^i}{i!}T_{-c-r}^{[i]}(E^ig_{\bm}+iq^{m_2n_1}E^{i-1}e_{\bm+\bn})t^i\\
&&=\mbox{$\sum\limits_{i=0}^{\infty}$}\frac{(-1)^i}{i!}T_{-c-r}^{[i]}E^ig_{\bm}t^i
+q^{m_2n_1}\mbox{$\sum\limits_{i=0}^{\infty}$}\frac{(-1)^{i+1}}{i!}T_{-c-r}^{[i+1]}E^ie_{\bm+\bn}t^{i+1}\\
&&=J_{c+r}g_{\bm}-q^{m_2n_1}J_{c+r}e_{\bm+\bn}T_{1-c}t,\\
h_{\bm}J_c\!\!\!\!\!\!\!&&=\mbox{$\sum\limits_{i=0}^{\infty}$}\frac{(-1)^i}{i!}h_{\bm}T_{-c}^{[i]}E^it^i
=\mbox{$\sum\limits_{i=0}^{\infty}$}\frac{(-1)^i}{i!}T_{-c-r}^{[i]}h_{\bm}E^it^i\\
&&=\mbox{$\sum\limits_{i=0}^{\infty}$}\frac{(-1)^i}{i!}T_{-c-r}^{[i]}(E^ih_{\bm}-iq^{m_1n_2}E^{i-1}e_{\bm+\bn})t^i\\
&&=\mbox{$\sum\limits_{i=0}^{\infty}$}\frac{(-1)^i}{i!}T_{-c-r}^{[i]}E^ih_{\bm}t^i
-q^{m_1n_2}\mbox{$\sum\limits_{i=0}^{\infty}$}\frac{(-1)^{i+1}}{i!}T_{-c-r}^{[i+1]}E^ie_{\bm+\bn}t^{i+1}\\
&&=J_{c+r}h_{\bm}+q^{m_1n_2}J_{c+r}e_{\bm+\bn}T_{1-c}t,\\
dJ_c\!\!\!\!\!\!\!&&=\mbox{$\sum\limits_{i=0}^{\infty}$}\frac{(-1)^i}{i!}dT_{-c}^{[i]}E^it^i
=\mbox{$\sum\limits_{i=0}^{\infty}$}\frac{(-1)^i}{i!}T_{-c}^{[i]}dE^it^i
=\mbox{$\sum\limits_{i=0}^{\infty}$}\frac{(-1)^i}{i!}T_{-c}^{[i]}(E^id+2iE^i)t^i\\
&&=\mbox{$\sum\limits_{i=0}^{\infty}$}\frac{(-1)^i}{i!}T_{-c}^{[i]}E^idt^i
+2\mbox{$\sum\limits_{i=0}^{\infty}$}\frac{(-1)^{i+1}}{i!}T_{-c}^{[i+1]}E^{i+1}t^{i+1}
=J_cd-2J_cET_{1-c}t.
\end{eqnarray*}
\end{proof}
\end{lemm}

\begin{proof}[Proof of Theorem \ref{main}\,(2)]
Using equations \eqref{equa-D-D0}, \eqref{equa-II-JJ} and the lemmas from Lemma \ref{lemm-4-1} to Lemma \ref{lemm-4-3}, for $\bm\in\BZ$, there are,
\begin{eqnarray*}
&\D(e_{\bm})\!\!\!\!&=\II \D_0(e_{\bm}) \II^{-1}=\II(e_{\bm}\otimes1+1\otimes e_{\bm})I=\II I_{-r}(e_{\bm}\otimes1)+\II I(1\otimes e_{\bm})\\
&&=e_{\bm}\otimes(1-Et)^r+1\otimes e_{\bm}.\\
&\D(g_{\bm})\!\!\!\!&=\II \D_0(g_{\bm}) \II^{-1}=\II(g_{\bm}\otimes1+1\otimes g_{\bm})I\\
&&=\II I_{-r}(g_{\bm}\otimes1)+\II(I(1\otimes g_{\bm})+q^{m_2n_1}I_{1}(T\otimes e_{\bm+\bn}t))\\
&&=g_{\bm}\otimes(1-Et)^r+
1\otimes g_{\bm}+q^{m_2n_1}T\otimes (1-Et)^{-1}e_{\bm+\bn}t,\\
&\D(h_{\bm})\!\!\!\!&=\II \D_0(h_{\bm}) \II^{-1}=\II(h_{\bm}\otimes1+1\otimes h_{\bm})I\\
&&=\II I_{-r}(h_{\bm}\otimes1)+\II(I(1\otimes h_{\bm})-q^{m_1n_2}I_{1}(T\otimes e_{\bm+\bn}t))\\
&&=h_{\bm}\otimes(1-Et)^r+
1\otimes h_{\bm}-q^{m_1n_2}T\otimes (1-Et)^{-1}e_{\bm+\bn}t,\\
&\D(d)\!\!\!\!&=\II \D_0(d) \II^{-1}=\II(d\otimes1+1\otimes d)I\\
&&=\II I(d\otimes1)+\II(I(1\otimes d)+2I_{1}(T\otimes Et
))\\
&&=d\otimes1+1\otimes d+2T\otimes (1-Et)^{-1}Et,\\
&\D(d_1)\!\!\!\!&=\II \D_0(d_1) \II^{-1}=\II(d_1\otimes1+1\otimes d_1)I\\
&&=\II I(d_1\otimes1)+\II (n_1I_{1}(T\otimes Et)+I(1\otimes d_1))\\
&&=d_1\otimes1+1\otimes d_1-n_1T\otimes1+n_1T\otimes(1-Et)^{-1},\\
&\D(d_2)\!\!\!\!&=\II \D_0(d_2) \II^{-1}=\II(d_2\otimes1+1\otimes d_2)I\\
&&=\II I(d_2\otimes1)+\II (n_2I_{1}(T\otimes Et)+I(1\otimes d_2))\\
&&=d_2\otimes1+1\otimes d_2-n_2T\otimes1+n_2T\otimes(1-Et)^{-1},\\
&S(e_{\bm})\!\!\!\!&=\JJ s_0(e_{\bm}) J=-\JJ e_{\bm}J=-\JJ J_r(e_{\bm})=-(1-Et)^{r}e_{\bm},\\
&S(g_{\bm})\!\!\!\!&=\JJ s_0(g_{\bm}) J=-\JJ (g_{\bm})J=-\JJ (J_{r}g_{\bm}-q^{m_2n_1}J_{r}(e_{\bm+\bn}T_{1}t))\\
&&=-(1-Et)^{r}g_{\bm}+q^{m_2n_1}(1-Et)^{r}e_{\bm+\bn}T_1t,\\
&S(h_{\bm})\!\!\!\!&=\JJ s_0(h_{\bm}) J=-\JJ (h_{\bm})J=-\JJ (J_{r}h_{\bm}+q^{m_1n_2}J_{r}(e_{\bm+\bn}T_{1}t))\\
&&=-(1-Et)^{r}h_{\bm}-q^{m_1n_2}(1-Et)^{r}e_{\bm+\bn}T_{1}t,\\
&S(d)\!\!\!\!&=\JJ s_0(d) J=-\JJ d J=-\JJ (Jd-2J(ET_{1}t))=-d+2ET_1t,\\
&S(d_1)\!\!\!\!&=\JJ s_0(d_1) J=-\JJ d_1 J=-\JJ (Jd_1-n_1JTEt)=-d_1+n_1TEt,\\
&S(d_2)\!\!\!\!&=\JJ s_0(d_2) J=-\JJ d_2 J=-\JJ (Jd_2-n_2JTEt)=-d_2+n_2TEt.
\end{eqnarray*}
Moreover, the following equations are also necessary to the theorem, where $\bm+\bn\neq0$,
\begin{eqnarray*}
&\D(f_{\bm})\!\!\!\!&=\II \D_0(f_{\bm}) \II^{-1}=\II(f_{\bm}\otimes1+1\otimes f_{\bm})I\\
&&=\II I_{-r}(f_{\bm}\otimes1)+\II I(1\otimes f_{\bm})-q^{m_1n_2}\II I_{1}(T\otimes g_{\bm+\bn}t)\\
&&\ \ \ \ +q^{m_2n_1}\II I_{1}(T\otimes h_{\bm+\bn}t)-s_{\bm}
\II I_{2}(T^{<2>}\otimes e_{\bm+2\bn}t^2)\\
&&=f_{\bm}\otimes(1-Et)^r+
1\otimes f_{\bm}-q^{m_1n_2}T\otimes (1-Et)^{-1}g_{\bm+\bn}t\\
&&\ \ \ \ +q^{m_2n_1}T\otimes (1-Et)^{-1}h_{\bm+\bn}t-s_{\bm}T^{<2>}\otimes (1-Et)^{-2}e_{\bm+2\bn}t^2,\\
&\D(f_{-\bn})\!\!\!\!&=\II \D_0(f_{-\bn}) \II^{-1}=\II(f_{-\bn}\otimes1+1\otimes f_{-\bn})I\\
&&=\II I_{1}(f_{-\bn}\otimes1)+\II(I(1\otimes f_{-\bn})-q^{-n_2n_1}I_{1}(T\otimes dt)-q^{-n_1n_2}I_{2}(T^{<2>}\otimes Et^2))\\
&&=f_{-\bn}\otimes(1-Et)^{-1}+
1\otimes f_{-\bn}-q^{-n_2n_1}T\otimes (1-Et)^{-1}dt\\
&&\ \ \ \ -q^{-n_1n_2}T^{<2>}\otimes (1-Et)^{-2}Et^2,\\
&S(f_{\bm})\!\!\!\!&=\JJ s_0(f_{\bm}) J=-\JJ f_{\bm}J\\
&&=\JJ J_{r}(q^{m_2n_1}h_{\bm+\bn}T_{1}t-q^{m_1n_2}g_{\bm+\bn}T_{1}t
-f_{\bm}+s_{\bm}e_{\bm+2\bn}T_{1}^{<2>}t^2)\\
&&=q^{m_2n_1}(1-Et)^{r}h_{\bm+\bn}T_{1}t
-q^{m_1n_2}(1-Et)^{r}g_{\bm+\bn}T_{1}t\\
&&\ \ \ \ -(1-Et)^{r}f_{\bm}+s_{\bm}(1-Et)^{r}e_{\bm+2\bn}T_{1}^{<2>}t^2,\\
&S(f_{-\bn})\!\!\!\!&=\JJ s_0(f_{-\bn}) J=-\JJ f_{-\bn}J\\
&&=-\JJ(J_{-1}f_{-\bn}+q^{-n_1n_2}J_{-1}(dT_{1}t)-q^{-n_1n_2}J_{-1}ET_{1}^{<2>}t^2)\\
&&=-(1-Et)^{-1}f_{-\bn}-q^{-n_1n_2}(1-Et)^{-1}dT_{1}t+q^{-n_1n_2}(1-Et)^{-1}ET_1^{<2>}t^2.
\end{eqnarray*}
\end{proof}

\vskip12pt

\noindent{\bf5. Proof of Theorem\ref{main}\,(3)}\setcounter{section}{5}\setcounter{theo}{0}\setcounter{equation}{0}

\vskip6pt

In this section, we take $T=\frac{1}{2}d$ and $E=e_{\bn}$ for $\bn=(n_1,n_2)\in\BZ$ such that $[T,E]=E$. The expressions only referring to $E$ in Section 4 are also tenable in this section, such as expressions from \eqref{Td12Ee-Id1} to \eqref{Td12Eg-Id} in Lemma \ref{lemm-4-2}.
For $\bm=(m_1,m_2)\in\BZ$, $r=x_1m_1+x_2m_2$, denote
\begin{eqnarray*}
s_{\bm}=q^{n_2m_1+n_1m_2+n_1n_2}.
\end{eqnarray*}

\begin{lemm}\label{lemm-5-1}
The following identities hold in $\UU(\wsl)$:
\begin{eqnarray}
&e_{\bm}T_c^{[i]}=T_{c-1}^{[i]}e_{\bm},\ \ &e_{\bm}T_c^{<i>}=T_{c-1}^{<i>}e_{\bm},\label{TdEe-Te}\\
&f_{\bm}T_c^{[i]}=T_{c+1}^{[i]}f_{\bm},\ \ &f_{\bm}T_c^{<i>}=T_{c+1}^{<i>}f_{\bm},\label{TdEe-Tf}\\
&g_{\bm}T_c^{[i]}=T_{c}^{[i]}g_{\bm},\ \ \ &g_{\bm}T_c^{<i>}=T_{c}^{<i>}g_{\bm},\label{TdEe-Tg}\\
&h_{\bm}T_c^{[i]}=T_{c}^{[i]}h_{\bm},\ \ &h_{\bm}T_c^{<i>}=T_{c}^{<i>}h_{\bm},\label{TdEe-Th}\\
&dT_c^{[i]}=T_{c}^{[i]}d,\ \ \ \ &dT_c^{<i>}=T_{c}^{<i>}d,\label{TdEe-Td}\\
&d_1T_c^{[i]}=T_{c}^{[i]}d_1,\ \ &d_1T_c^{<i>}=T_{c}^{<i>}d_1,\label{TdEe-Td1}\\
&d_2T_c^{[i]}=T_{c}^{[i]}d_2,\ \ &d_2T_c^{<i>}=T_{c}^{<i>}d_2.\label{TdEe-Td2}
\end{eqnarray}
\end{lemm}
\begin{proof}~Since $[T,e_{\bm}]=e_{\bm}$, $[T,f_{\bm}]=-f_{\bm}$, $[T,g_{\bm}]=[T,h_{\bm}]=[T,d]=[T,d_1]=[T,d_2]=0$, we obtain equations from \eqref{TdEe-Te} to \eqref{TdEe-Td2}.
\end{proof}

\begin{lemm}\label{lemm-5-2}
The following identities hold in $\UU(\wsl)$:
\begin{eqnarray}
&&(e_{\bm}\otimes1)I_c=I_{c-1}(e_{\bm}\otimes1),\label{TdEe-eI}\\
&&(f_{\bm}\otimes1)I_c=I_{c+1}(f_{\bm}\otimes1),\label{TdEe-fI}\\
&&(g_{\bm}\otimes1)I_c=I_{c}(g_{\bm}\otimes1),\label{TdEe-gI}\\
&&(h_{\bm}\otimes1)I_c=I_{c}(h_{\bm}\otimes1),\label{TdEe-hI}\\
&&(d\otimes1)I_c=I_{c}(d\otimes1),\label{TdEe-dI}\\
&&(d_1\otimes1)I_c=I_{c}(d_1\otimes1),\label{TdEe-d1I}\\
&&(d_2\otimes1)I_c=I_{c}(d_2\otimes1).\label{TdEe-d2I}
\end{eqnarray}
\end{lemm}
\begin{proof} The equation \eqref{TdEe-eI} can be obtained as following by equation \eqref{TdEe-Te},
\begin{eqnarray*}
(e_{\bm}\otimes1)I_c=\mbox{$\sum\limits_{i=0}^{\infty}$}\frac{1}{i!}e_{\bm}T_c^{<i>}\otimes E^it^i=\mbox{$\sum\limits_{i=0}^{\infty}$}\frac{1}{i!}T_{c-1}^{<i>}e_{\bm}\otimes E^it^i=I_{c-1}(e_{\bm}\otimes1).
\end{eqnarray*}
By the similar method, we obtain equations from \eqref{TdEe-fI} to \eqref{TdEe-d2I}.
\end{proof}

\begin{lemm}
The following identities hold in $\UU(\wsl)$:
\begin{eqnarray}
&&e_{\bm}J_c=J_{c+1}e_{\bm},\label{TdEe-ej}\\
&&d_1J_c=J_cd_1-n_1J_cT_{-c}Et,\label{TdEe-d1j}\\
&&d_2J_c=J_cd_2-n_2J_cT_{-c}Et,\label{TdEe-d2j}\\
&&f_{\bm}J_c=\begin{cases}
q^{m_1n_2}J_{c-1}(T_{-c+1}g_{\bm+\bn}t)-q^{m_2n_1}J_{c-1}(h_{\bm+\bn}T_{1-c}t)\\
+J_{c-1}f_{\bm}-s_{\bm}J_{c-1}(e_{\bm+2\bn}T_{1-c}^{<2>}t^2),&\bm+\bn\neq0,\\ \\
J_{c-1}f_{-\bn}+q^{-n_1n_2}J_{c-1}(dT_{1-c}t)-q^{-n_1n_2}J_{c-1}ET_{1-c}^{<2>}t^2,&\bm+\bn=0,
\end{cases}\label{TdEe-fj}\\
&&g_{\bm}J_c=J_{c}g_{\bm}-q^{n_1m_2}J_{c}(e_{\bm+\bn}T_{1-c}t),\label{TdEe-gj}\\
&&h_{\bm}J_c=J_{c}h_{\bm}+q^{n_2m_1}J_{c}(e_{\bm+\bn}T_{1-c}t),\label{TdEe-hj}\\
&&dJ_c=J_cd-2J_c(ET_{1-c}t).\label{TdEe-dj}
\end{eqnarray}
\begin{proof}
Using the formula \eqref{equa-J}, \eqref{Td12Ee-deE} and \eqref{TdEe-Te}, we have
\begin{eqnarray*}
e_{\bm}J_c=\mbox{$\sum\limits_{i=0}^{\infty}$}\frac{(-1)^i}{i!}e_{\bm}T_{-c}^{[i]}E^it^i
=\mbox{$\sum\limits_{i=0}^{\infty}$}\frac{(-1)^i}{i!}T_{-c-1}^{[i]}e_{\bm}E^it^i
=\mbox{$\sum\limits_{i=0}^{\infty}$}\frac{(-1)^i}{i!}T_{-c-1}^{[i]}E^ie_{\bm}t^i=J_{c+1}e_{\bm}.
\end{eqnarray*}

Formulas \eqref{TdEe-d1j} and \eqref{TdEe-d2j} are the same as those presented in Lemma \ref{lemm-3-3}. For $\bm\neq-\bn$, there is
\begin{eqnarray*}
f_{\bm}J_c\!\!\!\!\!\!\!&&=\mbox{$\sum\limits_{i=0}^{\infty}$}\frac{(-1)^i}{i!}f_{\bm}T_{-c}^{[i]}E^it^i
=\mbox{$\sum\limits_{i=0}^{\infty}$}\frac{(-1)^i}{i!}T_{-c+1}^{[i]}f_{\bm}E^it^i\\
&&=\mbox{$\sum\limits_{i=0}^{\infty}$}\frac{(-1)^i}{i!}T_{-c+1}^{[i]}(E^if_{\bm}\!\!-\!iq^{m_1n_2}E^{i-1}g_{\bm+\bn}
\!\!+iq^{m_2n_1}E^{i-1}h_{\bm+\bn}\!\!-\!2s_{\bm}\dbinom{i}{2}E^{i-2}e_{\bm+2\bn})t^i\\
&&=\mbox{$\sum\limits_{i=0}^{\infty}$}\frac{(-1)^i}{i!}T_{-c+1}^{[i]}E^if_{\bm}t^i-q^{m_1n_2}\mbox{$\sum\limits_{i=0}^{\infty}$}\frac{(-1)^{i+1}}{i!}T_{-c+1}^{[i+1]}E^{i}g_{\bm+\bn}t^{i+1}\\
&&\ \ \ \ +q^{m_2n_1}\mbox{$\sum\limits_{i=0}^{\infty}$}\frac{(-1)^{i+1}}{i!}T_{-c+1}^{[i+1]}E^{i}h_{\bm+\bn}t^{i+1}-s_{\bm}\mbox{$\sum\limits_{i=0}^{\infty}$}\frac{(-1)^{i+2}}{i!}T_{-c+1}^{[i+2]}E^{i}e_{\bm+2\bn}t^{i+2}\\
&&=J_{c-1}f_{\bm}+q^{m_1n_2}J_{c-1}(g_{\bm+\bn}T_{1-c}t)-q^{m_2n_1}J_{c-1}(h_{\bm+\bn}T_{1-c}t)-s_{\bm}J_{c-1}(e_{\bm+2\bn}T_{1-c}^{<2>}t^2).
\end{eqnarray*}
Furthermore, there is
\begin{eqnarray*}
f_{-\bn}J_c\!\!\!\!\!\!\!&&=\mbox{$\sum\limits_{i=0}^{\infty}$}\frac{(-1)^i}{i!}f_{-\bn}T_{-c}^{[i]}E^it^i
=\mbox{$\sum\limits_{i=0}^{\infty}$}\frac{(-1)^i}{i!}T_{-c+1}^{[i]}f_{-\bn}E^it^i\\
&&=\mbox{$\sum\limits_{i=0}^{\infty}$}\frac{(-1)^i}{i!}T_{-c+1}^{[i]}(E^if_{-\bn}-iq^{-n_2n_1}E^{i-1}d
-2q^{-n_2n_1}\dbinom{i}{2}E^{i-1})t^i\\
&&=\mbox{$\sum\limits_{i=0}^{\infty}$}\frac{(-1)^i}{i!}T_{-c+1}^{[i]}E^if_{-\bn}t^i
-q^{-n_2n_1}\mbox{$\sum\limits_{i=0}^{\infty}$}\frac{(-1)^{i+1}}{i!}T_{-c+1}^{[i+1]}E^{i}dt^{i+1}\\
&&\ \ \ \ -q^{-n_1n_2}\mbox{$\sum\limits_{i=0}^{\infty}$}\frac{(-1)^{i+2}}{i!}T_{-c+1}^{[i+2]}E^{i+1}t^{i+2}\\
&&=J_{c-1}f_{-\bn}+q^{-n_2n_1}J_{c-1}T_{1-c}dt-q^{-n_1n_2}J_{c-1}T_{1-c}^{[2]}Et^2\\
&&=J_{c-1}f_{-\bn}+q^{-n_1n_2}J_{c-1}dT_{1-c}t-q^{-n_1n_2}J_{c-1}ET_{1-c}^{<2>}t^2.
\end{eqnarray*}
Thus, we obtain \eqref{TdEe-fj}. Equations \eqref{TdEe-gj}  to \eqref{TdEe-dj} could be obtained by the formulas \eqref{x<>[]}, \eqref{equa-J} and formulas from \eqref{TdEe-Tg} to \eqref{TdEe-Td},
\begin{eqnarray*}
g_{\bm}J_c\!\!\!\!\!\!\!&&=\mbox{$\sum\limits_{i=0}^{\infty}$}\frac{(-1)^i}{i!}g_{\bm}T_{-c}^{[i]}E^it^i
=\mbox{$\sum\limits_{i=0}^{\infty}$}\frac{(-1)^i}{i!}T_{-c}^{[i]}g_{\bm}E^it^i\\
&&=\mbox{$\sum\limits_{i=0}^{\infty}$}\frac{(-1)^i}{i!}T_{-c}^{[i]}(E^ig_{\bm}+iq^{m_2n_1}E^{i-1}e_{\bm+\bn})t^i\\
&&=\mbox{$\sum\limits_{i=0}^{\infty}$}\frac{(-1)^i}{i!}T_{-c}^{[i]}E^ig_{\bm}t^i
+q^{m_2n_1}\mbox{$\sum\limits_{i=0}^{\infty}$}\frac{(-1)^{i+1}}{i!}T_{-c}^{[i+1]}E^ie_{\bm+\bn}t^{i+1}\\
&&=J_{c}g_{\bm}-q^{m_2n_1}J_{c}e_{\bm+\bn}T_{1-c}t,\\
h_{\bm}J_c\!\!\!\!\!\!\!&&=\mbox{$\sum\limits_{i=0}^{\infty}$}\frac{(-1)^i}{i!}h_{\bm}T_{-c}^{[i]}E^it^i
=\mbox{$\sum\limits_{i=0}^{\infty}$}\frac{(-1)^i}{i!}T_{-c}^{[i]}h_{\bm}E^it^i\\
&&=\mbox{$\sum\limits_{i=0}^{\infty}$}\frac{(-1)^i}{i!}T_{-c}^{[i]}(E^ih_{\bm}-iq^{m_1n_2}E^{i-1}e_{\bm+\bn})t^i\\
&&=\mbox{$\sum\limits_{i=0}^{\infty}$}\frac{(-1)^i}{i!}T_{-c}^{[i]}E^ih_{\bm}t^i
-q^{m_1n_2}\mbox{$\sum\limits_{i=0}^{\infty}$}\frac{(-1)^{i+1}}{i!}T_{-c}^{[i+1]}E^ie_{\bm+\bn}t^{i+1}\\
&&=J_{c}h_{\bm}+q^{m_1n_2}J_{c}e_{\bm+\bn}T_{1-c}t,\\
dJ_c\!\!\!\!\!\!\!&&=\mbox{$\sum\limits_{i=0}^{\infty}$}\frac{(-1)^i}{i!}dT_{-c}^{[i]}E^it^i
=\mbox{$\sum\limits_{i=0}^{\infty}$}\frac{(-1)^i}{i!}T_{-c}^{[i]}dE^it^i
=\mbox{$\sum\limits_{i=0}^{\infty}$}\frac{(-1)^i}{i!}T_{-c}^{[i]}(E^id+2iE^i)t^i\\
&&=\mbox{$\sum\limits_{i=0}^{\infty}$}\frac{(-1)^i}{i!}T_{-c}^{[i]}E^idt^i
+2\mbox{$\sum\limits_{i=0}^{\infty}$}\frac{(-1)^{i+1}}{i!}T_{-c}^{[i+1]}E^{i+1}t^{i+1}
=J_cd-2J_cET_{1-c}t.
\end{eqnarray*}
\end{proof}
\end{lemm}

\begin{proof}[Proof of Theorem \ref{main}\,(3)]
Using equations \eqref{equa-D-D0}, \eqref{equa-II-JJ} and all the lemmas above on this section, we obtain, for $\bm+\bn\neq 0$,
\begin{eqnarray*}
&\D(f_{\bm})\!\!\!\!&=\II \D_0(f_{\bm}) \II^{-1}=\II(f_{\bm}\otimes1+1\otimes f_{\bm})I\\
&&=\II I_{1}(f_{\bm}\otimes1)+\II I(1\otimes f_{\bm})-q^{m_1n_2}\II I_{1}(T\otimes g_{\bm+\bn}t)\\
&&\ \ \ \ +q^{m_2n_1}\II I_{1}(T\otimes h_{\bm+\bn}t)-s_{\bm}\II I_{2}(T^{<2>}\otimes e_{\bm+2\bn}t^2)\\
&&=f_{\bm}\otimes(1-Et)^{-1}
+1\otimes f_{\bm}-q^{m_1n_2}T\otimes (1-Et)^{-1}g_{\bm+\bn}t\\
&&\ \ \ \ +q^{m_2n_1}T\otimes (1-Et)^{-1}h_{\bm+\bn}t-s_{\bm}T^{<2>}\otimes (1-Et)^{-2}e_{\bm+2\bn}t^2,\\
&\D(f_{-\bn})\!\!\!\!&=\II \D_0(f_{-\bn}) \II^{-1}=\II(f_{-\bn}\otimes1+1\otimes f_{-\bn})I\\
&&=\II I_{1}(f_{-\bn}\otimes1)-q^{-n_2n_1}\II I_{1}(T\otimes dt)\\
&&\ \ \ \ +\II I(1\otimes f_{-\bn})-q^{-n_1n_2}\II I_{2}(T^{<2>}\otimes Et^2)\\
&&=f_{-\bn}\otimes(1-Et)^{-1}-q^{-n_2n_1}T\otimes(1-Et)^{-1} dt\\
&&\ \ \ \ +1\otimes f_{-\bn}-q^{-n_1n_2}T^{<2>}\otimes (1-Et)^{-2}t^2,\\
&S(f_{\bm})\!\!\!\!&=\JJ s_0(f_{\bm})J=-\JJ f_{\bm}J\\
&&=q^{m_2n_1}\JJ J_{-1}(h_{\bm+\bn}T_{1}t)-q^{m_1n_2}\JJ J_{-1}(T_{1}g_{\bm+\bn}t)\\
&&\ \ \ \ -\JJ J_{-1}f_{\bm}
+s_{\bm}\JJ J_{-1}(e_{\bm+2\bn}T_{1}^{<2>}t^2)\\
&&=q^{m_2n_1}(1-Et)^{-1}h_{\bm+\bn}T_{1}t
-q^{m_1n_2}(1-Et)^{-1}T_1g_{\bm+\bn}t\\
&&\ \ \ \ -(1-Et)^{-1}f_{\bm}
+s_{\bm}(1-Et)^{-1}e_{\bm+2\bn}T_{1}^{<2>}t^2,\\
&S(f_{-\bn})\!\!\!\!&=\JJ s_0(f_{-\bn})J=-\JJ f_{-\bn}J\\
&&=-\JJ(J_{-1}f_{-\bn}+q^{-n_1n_2}J_{-1}(dT_{1}t)-q^{-n_1n_2}J_{-1}ET_{1}^{<2>}t^2)\\
&&=-(1-Et)^{-1}f_{-\bn}-q^{-n_1n_2}(1-Et)^{-1}dT_1t+q^{-n_1n_2}(1-Et)^{-1}ET_{1}^{<2>}t^2.
\end{eqnarray*}
What is more, for any $\bm\in\BZ$, we also obtain,
\begin{eqnarray*}
&\D(e_{\bm})\!\!\!\!&=\II \D_0(e_{\bm}) \II^{-1}=\II(e_{\bm}\otimes1+1\otimes e_{\bm})I\\
&&=\II I_{-1}(e_{\bm}\otimes1)+\II I(1\otimes e_{\bm})\\
&&=e_{\bm}\otimes(1-Et)+1\otimes e_{\bm},\\
&\D(g_{\bm})\!\!\!\!&=\II \D_0(g_{\bm}) \II^{-1}=\II(g_{\bm}\otimes1+1\otimes g_{\bm})I\\
&&=\II I(g_{\bm}\otimes1)+\II(I(1\otimes g_{\bm})+q^{m_2n_1}I_{1}(T\otimes e_{\bm+\bn}t))\\
&&=g_{\bm}\otimes 1+
1\otimes g_{\bm}+q^{m_2n_1}T\otimes (1-Et)^{-1}e_{\bm+\bn}t,\\
&\D(h_{\bm})\!\!\!\!&=\II \D_0(h_{\bm}) \II^{-1}=\II(h_{\bm}\otimes1+1\otimes h_{\bm})I\\
&&=\II I_{0}(h_{\bm}\otimes1)+\II(I(1\otimes h_{\bm})-q^{m_1n_2}I_{1}(T\otimes e_{\bm+\bn}t))\\
&&=h_{\bm}\otimes 1+
1\otimes h_{\bm}-q^{m_1n_2}T\otimes (1-Et)^{-1}e_{\bm+\bn}t,\\
&\D(d)\!\!\!\!&=\II \D_0(d) \II^{-1}=\II(d\otimes1+1\otimes d)I\\
&&=\II I(d\otimes1)+\II(I(1\otimes d)+2I_{1}(T\otimes Et
))\\
&&=d\otimes1+1\otimes d+2T\otimes (1-Et)^{-1}Et,\\
&\D(d_1)\!\!\!\!&=\II \D_0(d_1) \II^{-1}=\II(d_1\otimes1+1\otimes d_1)I\\
&&=\II I(d_1\otimes1)+\II (n_1I_{1}(T\otimes Et)+I(1\otimes d_1))\\
&&=d_1\otimes1+1\otimes d_1+n_1T\otimes(1-Et)^{-1}-n_1T\otimes1,\\
&\D(d_2)\!\!\!\!&=\II \D_0(d_2) \II^{-1}=\II(d_2\otimes1+1\otimes d_2)I\\
&&=\II I(d_2\otimes1)+\II (n_2I_{1}(T\otimes Et)+I(1\otimes d_2))\\
&&=d_2\otimes1+1\otimes d_2+n_2T\otimes(1-Et)^{-1}-n_2T\otimes1.\\
&S(e_{\bm})\!\!\!\!&=\JJ s_0(e_{\bm})J=-\JJ e_{\bm}J=-\JJ J_{1}e_{\bm}=-(1-Et)e_{\bm},\\
&S(g_{\bm})\!\!\!\!&=\JJ s_0(g_{\bm})J=-\JJ g_{\bm}J=-\JJ(Jg_{\bm}-q^{m_2n_1}J(e_{\bm+\bn}T_{1}t))\\
&&=-g_{\bm}+q^{m_2n_1}e_{\bm+\bn}T_{1}t,\\
&S(h_{\bm})\!\!\!\!&=\JJ s_0(h_{\bm})J=-\JJ h_{\bm}J=-\JJ(Jh_{\bm}+q^{n_2m_1}J(e_{\bm+\bn}T_{1}t))\\
&&=-h_{\bm}-q^{n_2m_1}e_{\bm+\bn}T_{1}t,\\
&S(d)\!\!\!\!&=\JJ s_0(d)J=-\JJ dJ=-\JJ(Jd-2J(ET_{1}t))=-d+2ET_1t,\\
&S(d_1)\!\!\!\!&=\JJ s_0(d_1)J=-\JJ d_1J=-\JJ(Jd_1-n_1JTEt)
=-d_1+n_1TEt,\\
&S(d_2)\!\!\!\!&=\JJ s_0(d_2)J=-\JJ d_2J=-\JJ(Jd_2-n_2JTEt)
=-d_2+n_2TEt.
\end{eqnarray*}
\end{proof}

\noindent{\bf Acknowledgements}\ \ The authors would sincerely like to thank Professor Yucai Su for his supervision and invaluable comments.

\end{CJK*}

\vskip8pt

\small
\begin{thebibliography}{9999}\parskip2pt\lineskip4pt

\bibitem{AG} B. Allison, Y. Gao, The root system and the core of an extended
affine Lie algebra, {\it Selecta Math. (N.S.)}, {\bf 7}(2) (2001), 149--212.

\bibitem{AABGP} B. Allison, S. Azam, S. Berman, Y. Gao, A. Pianzola, Extended affine Lie algebras and their root systems, {\it Mem. Amer. Math. Soc.}, {\bf 126}(603) (1997).

\bibitem{B} Y. Billig, Representations of toroidal extended affine Lie algebras, {\it J. Algebra}, {\bf 308} (2007), 252--269.

\bibitem{BGK} S. Berman, Y. Gao, Y. Krylyuk, Quantum tori and the structure of elliptic quasi-simple Lie algebras, {\it J. Funct. Anal.}, {\bf 135} (1996) 339--389.

\bibitem{BGKE} S. Berman, Y. Gao, Y. Krylyuk,  E. Neher, The alternative torus and the structure of elliptic quasisimple Lie algebras of type $A_2$, {\it Trans. Amer. Math. Soc.}, {\bf 347}(11) (1995), 4315--4363.

\bibitem{BL} Y. Billig, M. Lau, Irreducible modules for extended affine Lie algebras, {\it J. Algebra}, {\bf 327} (2011) 208--235.

\bibitem{CS}Y. Cheng, Y. Su, Quantization of Lie algebras of block type, {\it Acta Mathematica Scientia}, {\bf 30}(4), (2010), 1134--1142.

\bibitem{D} V. Drinfel'd, Constant quasiclassical solutions of the
Yang-Baxter quantum equation, {\it Soviet Math. Dokl.}, {\bf28}(3) (1983),
667--671.

\bibitem{D1} V. Drinfel'd, Quantum groups, Proceedings ICM (Berkeley 1986),
Providence: Amer Math Soc, (1987) 789--820.

\bibitem{D2} V. Drinfel'd, Quantum groups, in: {\it Proceeding of the
International Congress of Mathematicians}, Vol. 1, 2, Berkeley,
Calif. 1986, Amer. Math. Soc., Providence, RI, 1987, 798--820.

\bibitem{EK} P. Etingof, D. Kazhdan, Quantization of Lie bialgebras I,
{\it Selecta Math.}, 2(1) (1996) 1--41.

\bibitem{EK2} P. Etingof, D. Kazhdan, Quantization of Lie bialgebras II,
{\it Selecta Math.}, 2(4) (1998) 213--231.

\bibitem{ES} P. Etingof, O. Schiffmann, Lectures on Quantum Groups, 2ed,
International Press, USA, 2002.

\bibitem{G1} Y. Gao, Representations of extended affine Lie algebras coordinatized by certain quantum tori, {\it Compositio Math.}, {\bf 123}(1) (2000), 1--25\vs{-7pt}.

\bibitem{G2} Y. Gao, Fermionic and bosonic representations of the extended affine Lie algebra \mbox{\footnotesize$\widetilde{\frak{gl}_N(\mathbb{C}_q)}$}, {\it Cananian Math. Bull.}, {\bf 45} (2002), 623--633\vs{-7pt}.

\bibitem{GZ} Y. Gao, Z. Zeng, Hermitian representations of the extended affine Lie algebra \mbox{\footnotesize$\widetilde{\frak{gl}_2(\mathbb{C}_q)}$}, {\it Adv. Math.}, {\bf 207}(1) (2006), 244--265.

\bibitem{G} C. Grunspan, Quantizations of the Witt algebra and of simple
Lie algebras in characteristic $p$, J. Algebra, 280 (2004), 145--161.

\bibitem{GZh} A. Giaquinto, J. Zhang, Bialgebra action, twists and universal deformation formulas, J. Pure Appl. Algebra, 128(2) (1998), 133--151.

\bibitem{HT} R. H{\o}egh-Krohn, B. Torresani, Classification and construction of quasi-simple Lie algebras, {\it J. Funct. Anal.}, {\bf 89} (1990), 106--136.

\bibitem{JM} C. Jiang, D. Meng, The derivation algebra of the associative algebra $C_q[X,Y,X^{-1},Y^{-1}]$, {\it Comm. Algebra}, {\bf 26}(6) (1998), 1723--1736.

\bibitem{LS}J. Li, Y. Su, Quantizations of the W-algebra W(2, 2), {\it Acta Mathematica Sinica}, {\bf 27}(4) (2011), 647--656.

\bibitem{LwS} W. Lin, Y. Su, Modules for the core of extended affine Lie algebras of type $A_1$ with coordinates in rank 2 quantum tori, {\it Pacific J. Math.}, {\bf 242}(1) (2009), 143--166.

\bibitem{N} E. Neher, Extended affine Lie algebras, {\it C. R. Math. Acad. Sci. Soc. R. Can.}, {\bf 26} (2004) 90--96.

\bibitem{SSW}G. Song, Y. Su, Y. Wu, Quantization of generalized Virasoro-like algebras, {\it Linear Algebra and its Applications}, {\bf 428}, (2008), 2888--2899.

\bibitem{XL}Y. Xu, J. Li, Lie bialgebra structures on the extended affine Lie algebra \wsl, arXiv:1102.5226v2.








\end{thebibliography}
\end{document}